\documentclass[10pt,draft,twoside]{article}

\usepackage{etex}                                               % Extensión de TeX
\usepackage[english]{babel}                                     % Idiomas: Español / Inglés                  % \selectlanguage{}
\usepackage{amsmath,amsthm,amssymb,amsfonts,latexsym,cancel}    % Paquetes de la American Mathematical Society (AMS)
\usepackage{enumerate}                                          % Enumerar listas
\usepackage{enumitem} \setlist[itemize]{noitemsep,topsep=0pt}   % Enumerar listas
\usepackage{pifont}                                             % Enumerar listas con símbolos especiales
\usepackage[latin1]{inputenc}                                   % Caracteres acentuados sin código
\usepackage{lmodern}                                            % Soporte de caracteres especiales en la fuente
\usepackage[T1]{fontenc}                                        % Evitar errores en letras acentuadas (ñ,á,è,...) en el PDF
\usepackage[none]{hyphenat}                                     % Evitar separar palabras con un guión al final del renglón
\usepackage{ragged2e} \justifying                               % Alineación y Justificación del texto
\usepackage{marvosym,manfnt,calligra}                           % Algunas fuentes tipo letra
\usepackage[mathscr]{euscript}                                  % Tipos de letras en modo matemático \mathscr
\usepackage{makeidx} \makeindex                                 % Para crear índice de palabras
\usepackage[dvips]{graphicx}                                    % Desplegar e imprimir imágenes en el documento
\usepackage{graphicx}                                           % Insertar gráficas e imágenes
\usepackage{color}                                              % Definir colores (texto coloreado, fondo coloreado, ...)
\usepackage{verbatim}                                           % Incluir texto sin formato latex
\usepackage{amssymb}                                            % Paquete de la AMS
\usepackage{amsfonts}                                           % Paquete de la AMS
\usepackage{amsmath}                                            % Paquete de la AMS
\usepackage{latexsym}                                           % Paquete de la AMS
\usepackage{amsthm}                                             % Paquete de la AMS
\usepackage{dsfont}                                             % Paquete de la AMS
\usepackage{float}                                              % Mejora la interfaz de objetos flotantes (figuras, gráficas, tablas, ...)
\usepackage{wrapfig}                                            % Figuras al lado de texto
\usepackage[rflt]{floatflt}                                     % Figuras flotantes entre el texto
\usepackage{nccmath}                                            % Ecuaciones a la izquierda o centradas    \begin{fleqn}...  \begin{ceqn}...
\usepackage{booktabs,longtable,multirow}%,slashbox}               % Crear e insertar Tablas
\usepackage[x11names,table]{xcolor}                             % Color en Tablas
\usepackage{multicol} 
\usepackage{tikz}                                               % Crear gráficas con TikZ
\usetikzlibrary{intersections}                                  % Permite que sean compatibles los paquetes "babel" y "pgf tikz"
\usetikzlibrary{mindmap,backgrounds,shapes,arrows,chains}       % Realizar mapas conceptuales, mapas mentales, diagramas de flujo, etc.
\usepackage[all,matrix,arrow]{xy}                               % Permite realizar diagramas conmutativos y de flechas
\usepackage{pgf,tikz} \usetikzlibrary{arrows}                   % Exportar de Geogebra: Vista Gráfica a PGF/TikZ
\usepackage{lipsum}
		
\newtheoremstyle{newstyle}   % nombre del estilo que se utilizará    
{} 			 				 % medida de espacio para dejar por encima del teorema
{12pt plus 1pt} 			 % medida de espacio para dejar por debajo del teorema
{\mdseries} 				 % nombre de fuente para usar en el cuerpo del teorema, e.g.\mdseries,\bfseries,\scshape,\itshape
{} %Indent					 % medida de espacio para sangrar
{\bfseries} 				 % nombre de fuente principal, e.g. \mdseries,\bfseries,\scshape,\itshape
{.} 						 % puntuación entre la cabeza y el cuerpo del teorema 
{ } 						 % espacio después de la cabeza del teorema
{} 							 % Especificar manualmente la cabeza del teorema

\theoremstyle{newstyle}
\newtheorem{definition}{Definition}[section]
\newtheorem{theorem}[definition]{Theorem}
\newtheorem{corollary}[definition]{Corollary}
\newtheorem{lemma}[definition]{Lemma}
\newtheorem{proposition}[definition]{Proposition}
\newtheorem{fact}[definition]{Fact}
\newtheorem{example}[definition]{Example}
\newtheorem{question}[definition]{Question}
\allowdisplaybreaks
\DeclareMathAlphabet{\mathcalligra}{T1}{calligra}{m}{n}
\DeclareMathAlphabet{\mathpzc}{OT1}{pzc}{m}{it}

\title{ \Large{\textbf{Some combinatorial properties of semiselective ideals}} }
\author{ \large{ Juli\'{a}n C. Cano $\;\;\;\;\;\;\;\;$ Carlos A. Di Prisco $\;\;\;\;\;\;\;\;$ Michael Hru\v{s}\'{a}k } }
\date{}
		
\newcommand{\Addresses}{{
		\bigskip
		\footnotesize
		
		Juli\'{a}n C. Cano, \textsc{Universidad de Los Andes (Bogot\'a).}\par\nopagebreak
		\textit{E-mail address}, J.C.~Cano: \texttt{jc.canor@uniandes.edu.co}
		
		\medskip
		
		Carlos A. Di Prisco, \textsc{Universidad de Los Andes (Bogot\'a),  Instituto Venezolano de Investigaciones Cient\'{\i}ficas (Caracas), and Universidad Nebrija (Madrid).}\par\nopagebreak
		\textit{E-mail address}, C.A. Di Prisco: \texttt{ca.di@uniandes.edu.co}
		
		\medskip
		
		Michael Hru\v{s}\'{a}k, \textsc{Universidad Nacional Aut\'onoma de M\'exico (Morelia).}\par\nopagebreak
		\textit{E-mail address}, M. Hru\v{s}\'{a}k: \texttt{michael@matmor.unam.mx}
		
}}
		
\begin{document}
			
\maketitle
			
\sloppy

\begin{abstract}
\noindent
We present several combinatorial properties of semiselective ideals on the set of natural numbers. The continuum hypothesis implies that the complement of every selective ideal contains a selective ultrafilter, however for semiselective ideals this is not the case. We prove that under certain hypothesis, for instance $V=L$, there are semiselective ideals whose complement does not contain a selective ultrafilter, and that it is also consistent that the complement of every semiselective ideal contains a selective ultrafilter; specifically, we show that if $V=L$ then there is a generic extension of $V$ where this occurs. We present other results concerning semiselective ideals, namely an alternative proof of Ellentuck's theorem for the local Ramsey property, and we prove some facts about the additivity of the ideal of local Ramsey null sets, and also about the generalized Suslin operation on the algebra of local Ramsey sets.

\medskip


\smallskip

\end{abstract}

\section{Introduction}

Ideals and filters of sets have been studied extensively from different points of view. Of particular interest for combinatorial set theory have been  the concepts of  selective ideal and Ramsey  ideal (see, for example, \cite{Hrusak, Uzcategui}). The Ramsey property  of sets of real numbers and its local version with respect to a selective ideal  was studied by Mathias in \cite{Mathias}, and also by Ellentuck, Farah, Matet, and Todorcevic, among other authors (see \cite{Ellentuck, Farah, Matet1, Todorcevic(BookTopology), Todorcevic(BookRamsey)}). Ellentuck in \cite{Ellentuck} gave a topological characterization of the Ramsey property that was later extended to the local Ramsey property with respect to a selective ideal. In this same vein, semiselective ideals were introduced in \cite{Farah}, where Farah showed not only that Ellentuck's characterization can be extended to the Ramsey property localized on a semiselective ideal, but also that semiselectivity is exactly what is needed of the ideal to get this topological characterization of the Ramsey property localized at an ideal, bringing to light  the importance of this concept.

\medskip

We will present here several results about the combinatorial structure of semiselective ideals on the set of natural numbers. We extend to semiselective ideals some results that were first proved for the ideal of finite sets of  natural numbers or for  selective ideals.

\medskip

This article is organized as follows. In Section \ref{semiselectiveideals} we review the notions of selective and semiselective ideals as well as the local Ramsey property with respect to an ideal. In Section \ref{newproof} we propose an alternative proof of  Ellentuck's theorem localized on a semiselective ideal. This proof is inspired on the short proof of Ellentuck's theorem given by Matet in \cite{Matet2} and avoids the use of the combinatorial forcing used in the proofs given in \cite{Farah, Todorcevic(BookRamsey)}. Section \ref{additivity} presents results related to the additivity of the ideal of sets that are Ramsey null with respect to a semiselective ideal, and also a characterization of semiselectivity is given in terms of closure under Suslin operation. 

\medskip

Other combinatorial properties of semiselective ideals are investigated in the final sections of the article; in particular, under which conditions  semiselective coideals contain selective ultrafilters. Similar problems for different types of ideals and filters have been studied by Matet and Pawlikowski in \cite{Matet-Pawlikowski}. It was shown by Mathias in \cite{Mathias} that if the continuum hypothesis holds, then every selective coideal contains a selective ultrafilter. We show in Section \ref{chandSuslin} that if the continuum hypothesis holds and there is a Suslin tree, then there exists a semiselective ideal whose complement does not contain a selective ultrafilter. In the final Section \ref{everysemiselective} we show that it is consistent that every semiselective coideal contains a selective ultrafilter. We get this consistency result via an iterated forcing construction producing a model where the cardinality of the continuum is $\mathfrak{c} = \aleph_2$. It remains open if it is possible to have that every semiselective coideal contains a selective ultrafilter together with $\mathfrak{c} = \aleph_1$.

\medskip

We use standard set theoretic notation. Given any set $X$ and any cardinal $\lambda$, we denote by $[X]^{\lambda}$ the set of subsets of $X$ of cardinality $\lambda$; analogously, we denote by $[X]^{<\lambda}$ the set of subsets of $X$ of cardinality less than $\lambda$; in particular, $[X]^{\omega}$ is the set of countably infinite subsets of $X$, and $[X]^{<\omega}$ is the set of finite subsets of $X$. 

\medskip

For $s,t \in [\omega]^{<\omega}$ and $A\in [\omega]^{\omega}$, we write $s \sqsubset A$ if $s$ is an initial segment of $A$ in its increasing enumeration. Similarly, we write $s \sqsubseteq t$ if $s$ is an initial segment of $t$, and we put $s \sqsubset t$ whenever $s \sqsubseteq t$ and $s\neq t$. For $n\in \omega$ and $s\in [\omega]^{<\omega}$, if $s \neq \emptyset$ and $\max s <n$ then we just write $s<n$, and we always put $\emptyset <n$. Given a set $A\in [\omega]^{\omega}$, let $A/n = \{ m\in A : n<m \}$ for each $n\in \omega$; likewise, let $A/s = \{ m\in A : s < m \}$ for each $s\in [\omega]^{<\omega}$.

\medskip

For every $s\in [\omega]^{<\omega}$, let $[s] = \{ B\in [\omega]^{\omega} : s \sqsubset B \}$, and notice that the family $\left \{ [s] : s\in [\omega]^{<\omega} \right \}$ forms a countable basis of closed--open sets for a Polish topology over $[\omega]^{\omega}$, called the \textit{metrizable topology}. Additionally, for every $s\in [\omega]^{<\omega}$ and every $A\in [\omega]^{\omega}$, we consider the set $[s,A]$ defined by $[s,A] = \{ B\in [s] : B/s \subseteq A \} = \{ B\in [\omega]^{\omega} : s \sqsubset B \subseteq s \cup A \}$. The collection $\mathsf{Exp} ([\omega]^{\omega}) = \{ [s,A] : s\in [\omega]^{<\omega} \wedge A\in [\omega]^{\omega} \}$ forms an uncountable basis of closed--open sets for a non--metrizable topology over $[\omega]^{\omega}$, called the \textit{Ellentuck topology}, which is finer than the metrizable topology. 

\medskip

Ramsey's theorem (see \cite{Ramsey}) states that for all $n,k \in \omega$ with $n,k\neq 0$, if $[\omega]^{n} = \bigcup_{i<k} G_{i}$ is any partition of $[\omega]^{n}$ into $k$ pieces, then there are $H\in [\omega]^{\omega}$ and $i<k$ such that $[H]^{n} \subseteq G_{i}$. In \cite{Ellentuck}, Ellentuck proved that if $[\omega]^{\omega} = \bigcup_{i<k} \mathcal{P}_{i}$ is any finite partition of $[\omega]^{\omega}$, where each piece $\mathcal{P}_{i}$ has the Baire property in the Ellentuck topology of $[\omega]^{\omega}$, then there is $H\in [\omega]^{\omega}$ such that $[H]^{\omega} \subseteq \mathcal{P}_{i}$ for some $i<k$.

\medskip

We say that a set $\mathcal{A} \subseteq [\omega]^{\omega}$ is \textit{Ramsey}\footnote{We adopt the terminology of \cite{Todorcevic(BookRamsey)}, referring to a set as \textit{Ramsey} when it has the combinatorial property often called \textit{completely Ramsey}.} if for every $[s,C]$ there is some $B\in [s,C]$ such that either $[s,B] \subseteq \mathcal{A}$ or $[s,B] \cap \mathcal{A} = \emptyset$; moreover, if the second option always holds, then we say that $\mathcal{A}$ is \textit{Ramsey null}. Erd\H os and Rado showed in \cite{Erdos-Rado}, using the axiom of choice, that there exist subsets of $[\omega]^{\omega}$ which are not Ramsey. Also, Galvin and Prikry proved in \cite{Galvin-Prikry} that the Borel sets in the metrizable topology of $[\omega]^{\omega}$ are Ramsey, and Silver in \cite{Silver} extended this fact to analytic sets. Additionally, Ellentuck in \cite{Ellentuck} characterized topologically Ramsey sets on $[\omega]^{\omega}$ as follows: let $\mathcal{A} \subseteq [\omega]^{\omega}$ be given, then $\mathcal{A}$ is Ramsey if and only if it has the Baire property with respect to  the Ellentuck topology, and $\mathcal{A}$ is Ramsey null if and only if it is meager with respect to the Ellentuck topology.

\section{Semiselective ideals and the local Ramsey property}\label{semiselectiveideals}

This section provides a brief introduction to selective and semiselective ideals and their interaction with the local Ramsey property, concluding with an exposition of the local Ellentuck theorem. The background for this section is drawn primarily from \cite[Section 7.1]{Todorcevic(BookRamsey)}.

\medskip
    
A family  $\mathcal{I}$ of subsets of an infinite set $X$ is an \textit{ideal} on $X$  if  $X \notin \mathcal{I}$,  $[X]^{<\omega} \subseteq \mathcal{I}$, and $\mathcal {I}$ is closed under taking subsets and finite unions of its elements. The complement of an ideal $\mathcal{I}$ on $X$ with respect to $\wp(X)$ is a \textit{coideal} and is denoted by $\mathcal{I}^{+}$, hence $\mathcal{I}^{+} = \wp(X) \setminus \mathcal{I}$. Given  $Y \in \mathcal{I}^{+}$, we denote by $\mathcal{I}^{+} \!\restriction\! Y$ the collection $\{ Z\in\mathcal{I}^{+} : Z \subseteq Y \}$. A \textit{non--principal ultrafilter} $\mathcal{U}$ on $X$ is basically a maximal coideal. 

\medskip

Fix an ideal $\mathcal{I}$ on $\omega$, and fix a quasi--order\footnote{A quasi--order is a binary relation that is reflexive and transitive.} relation $\preceq$ on $\mathcal{I}^{+}$; also, let $\mathcal{D} \subseteq \mathcal{I}^{+}$ be given. We say that $\mathcal{D}$ is \textit{dense} in $(\mathcal{I}^{+} , \preceq  )$ whenever for every $N \in \mathcal{I}^{+}$ there is $M \in \mathcal{D}$ such that $M \preceq  N$. We say that $\mathcal{D}$ is \textit{open} in $(\mathcal{I}^{+} , \preceq )$ whenever for each $N \in \mathcal{I}^{+}$ and each $M \in \mathcal{D}$, if $N \preceq  M$ then $N \in \mathcal{D}$.  We say that $\mathcal{D}$ is \textit{dense--open} in $(\mathcal{I}^{+} , \preceq )$ if $\mathcal{D}$ is both dense and open in $(\mathcal{I}^{+} , \preceq )$. For any ideal $\mathcal{I}$ on $\omega$, we will only use the quasi--orders on $\mathcal{I}^{+}$ given by the inclusion relation $\subseteq$ and the almost inclusion relation $\subseteq^{*}$. 

\medskip

The following is an interesting example of a dense--open set in $(\mathcal{I}^{+}, \subseteq)$, which will be implicitly used in the proof of Lemma \ref{Semiselective_Matet_Lemma}.

\begin{example}
	Let $\mathcal{I}$ be an ideal on $\omega$. Given any $t\in [\omega]^{<\omega}$ and $\mathcal{A}\subseteq [\omega]^{\omega}$, consider the set $\mathcal{S}_{t}^{\mathcal{A}} \subseteq [\omega]^{\omega}$ defined by $\mathcal{S}_{t}^{\mathcal{A}} = \{ D \in [\omega]^{\omega} : [t, D] \subseteq \mathcal{A} \}$.   Then, the set $\mathcal{E}_{t}^{\mathcal{A}} \subseteq \mathcal{I}^{+}$ defined by $\mathcal{E}_{t}^{\mathcal{A}} = \{ N\in\mathcal{I}^{+} : \mathcal{I}^{+} \!\restriction\! N \subseteq \mathcal{S}_{t}^{\mathcal{A}} \, \vee \, \mathcal{I}^{+} \!\restriction\! N \cap \mathcal{S}_{t}^{\mathcal{A}} = \emptyset \}$ is dense--open in $(\mathcal{I}^{+}, \subseteq)$.
\end{example}

We now present the well-known concept of selectivity for ideals, which will play an essential role in the final two sections of this article.

\begin{definition}[\cite{Mathias}]\label{Def_selective-ideal}
An ideal $\mathcal{I}$ on $\omega$ is \textit{selective} if for each $A\in \mathcal{I}^{+}$ and for every decreasing sequence $\{ {A}_{n} \}_{n\in \omega} \subseteq \mathcal{I}^{+} \!\restriction\! A$, there exists some $A_{\infty} \in \mathcal{I}^{+} \!\restriction\! A$ such that $A_{\infty}/m \subseteq {A}_{m}$ for all $m\in A_{\infty}$. In this case $A_\infty$ is said to be a diagonalization of the sequence  $\{ {A}_{n} \}_{n\in \omega}$.
\end{definition}

The prototypical example of a selective ideal on $\omega$ is the ideal $\mathcal{I}_{\mathcal{A}}$ generated by an almost disjoint family\footnote{An infinite family $\mathcal{A}$ of infinite subsets of $\omega$ is almost disjoint if $|X \cap Y|<\omega$ for every $X,Y\in\mathcal{A}$ with $X \neq Y$.} $\mathcal{A}$ on $\omega$ (see \cite[Example 7.1.2]{Todorcevic(BookRamsey)}), that is, $\mathcal{I}_{\mathcal{A}} = \{ X \subseteq \omega : (\exists\, \mathcal{F} \in [\mathcal{A}]^{<\omega}) (X \subseteq^{*} \bigcup \mathcal{F}) \}$. Selective coideals were introduced  by Mathias in \cite{Mathias} under the name of happy families. The following is a convenient characterization of selectivity (see \cite[Lemma 9.1]{Todorcevic(BookTopology)} and \cite[Lemma 7.4]{Todorcevic(BookRamsey)}).

\begin{proposition}[] \label{selective=(p)+(q)}
(Todorcevic, \cite{Todorcevic(BookTopology), Todorcevic(BookRamsey)}). Let $\mathcal I$ be an ideal on $\omega$. Then, the following statements are equivalent:
\begin{itemize}
\item[(a)] The ideal $\mathcal I$  is selective.
\item[(b)] The ideal $\mathcal I$ has the following two properties:
     \setlist{nolistsep}
	\begin{enumerate}
	\setlength{\itemsep}{0pt} 
		\item[$(p)$] For every $A\in \mathcal{I}^{+}$ and every decreasing sequence $\{A_{n}\}_{n\in\omega} \subseteq \mathcal{I}^{+} \!\restriction\! A$, there exists some $B\in \mathcal{I}^{+} \!\restriction\! A$ such that $B\subseteq^{*} A_{n}$ for all  $n\in \omega$.
		\item[$(q)$] For every $A\in \mathcal{I}^{+}$ and every infinite partition $A = \bigcup_{k\in\omega} f_{k}$, where each $f_{k}$ is finite, there exists some $B\in \mathcal{I}^{+} \!\restriction\! A$ such that for each $k\in \omega$ the set $f_{k}\cap B$ has at most one element. 
\end{enumerate}
\end{itemize}
\end{proposition}

Next, we present the concept of semiselectivity for ideals, a combinatorial notion that is weaker than selectivity. We emphasize that all results in this article focus on semiselectivity for ideals.

\begin{definition}[\cite{Farah}] \label{Def_Semiselective_Ideal}
	An ideal $\mathcal{I}$ on $\omega$ is \textit{semiselective} if for each $A\in \mathcal{I}^{+}$ and for every sequence $\{ \mathcal{D}_{n} \}_{n\in \omega} \subseteq \wp(\mathcal{I}^{+})$ of dense--open sets in $(\mathcal{I}^{+}, \subseteq)$, there exists some $D_{\infty} \in \mathcal{I}^{+} \!\restriction\! A$ such that $D_{\infty}/m \in \mathcal{D}_{m}$ for all $m\in D_{\infty}$. In this case $D_\infty$ is said to be a diagonalization of the sequence  $\{\mathcal {D}_{n} \}_{n\in \omega}$.
\end{definition}

Equivalently, an ideal $\mathcal{I}$ on $\omega$ is semiselective if for each $A\in \mathcal{I}^{+}$ and for every sequence $\{ \mathcal{D}_{s} \}_{s\in [\omega]^{<\omega}} \subseteq \wp(\mathcal{I}^{+})$ of dense--open sets in $(\mathcal{I}^{+}, \subseteq)$, there exists some $D_{\infty} \in \mathcal{I}^{+} \!\restriction\! A$ such that $D_{\infty}/t \in \mathcal{D}_{t}$ for all $t\in [D_{\infty}]^{<\omega}$. In this case $D_{\infty}$ is also said to be a diagonalization of the sequence $\{ \mathcal{D}_{s} \}_{s\in [\omega]^{<\omega}}$.

\medskip

It is easy to show that every selective ideal is semiselective, but the converse implication is not true in general. The ideal $\emptyset\times [\omega]^{<\omega}$ is a semiselective ideal which is not selective (see \cite[Example 2.1]{Farah}). In what follows, we will concentrate on combinatorial properties of semiselective ideals. The following proposition gives a useful characterization of semiselectivity (see \cite[Lemma 7.9]{Todorcevic(BookRamsey)}).

\begin{proposition}[] \label{semiselective=(p^w)+(q^+)}
(Todorcevic, \cite{Todorcevic(BookRamsey)}). Let $\mathcal I$ be an ideal on $\omega$. Then, the following statements are equivalent:
\begin{itemize}
\item[(a)] The ideal $\mathcal I$  is semiselective.

\item[(b)] The ideal $\mathcal I$  has the following two properties:
	\setlist{nolistsep}
	\begin{enumerate}
		\setlength{\itemsep}{0pt} 
		\item[$(p^{w})$] For every $A\in \mathcal{I}^{+}$ and every sequence $\{ \mathcal{D}_{n} \}_{n\in\omega} \subseteq \wp(\mathcal{I}^{+})$ of dense--open sets in $(\mathcal{I}^{+}, \subseteq)$, there exists some $B\in \mathcal{I}^{+} \!\restriction\! A$ with the property that for each $n\in\omega$ there is $A_{n}\in \mathcal{D}_{n}$ such that $B \subseteq^{*} A_{n}$.
		\item[$(q)$] For every $A\in \mathcal{I}^{+}$ and every infinite partition $A = \bigcup_{k\in\omega} f_{k}$, where each $f_{k}$ is finite, there exists some $B\in \mathcal{I}^{+} \!\restriction\! A$ such that for each $k\in \omega$ the set $f_{k}\cap B$ has at most one element. 
	\end{enumerate}
	\end{itemize}
\end{proposition}

Let $\mathcal{I}$ be an ideal on $\omega$, and the corresponding coideal $\mathcal{I}^{+} \subseteq [\omega]^{\omega}$. We consider the subcollection $\mathsf{Exp} (\mathcal{I}^{+}) \subseteq \mathsf{Exp}([\omega]^{\omega})$ of basic open sets in the Ellentuck topology of $[\omega]^{\omega}$ described by 
\begin{center}
	$\mathsf{Exp} (\mathcal{I}^{+}) = \{ [s,A] : s\in [\omega]^{<\omega} \wedge A\in \mathcal{I}^{+} \}$. 
\end{center}

Although the family $\mathsf{Exp} (\mathcal{I}^{+})$ is not necessarily a basis for a topology on $[\omega]^{\omega}$, it is still possible to work with some abstract topological notions defined from $\mathsf{Exp} (\mathcal{I}^{+})$ as follows:  
\setlist{nolistsep}
\begin{enumerate}
	\setlength{\itemsep}{0pt}	
	\item[I.] A set $\mathcal{X} \subseteq [\omega]^{\omega}$ is $\mathsf{Exp}(\mathcal{I}^{+})-$\textit{open} if $\mathcal{X}$ is the union of some arbitrary family of members of $\mathsf{Exp}(\mathcal{I}^{+})$.
	\item[II.] A set $\mathcal{X} \subseteq [\omega]^{\omega}$ is $\mathsf{Exp}(\mathcal{I}^{+})-$\textit{nowhere dense} if for every $[a,A] \in \mathsf{Exp}(\mathcal{I}^{+})$ there is some $[b,B] \in \mathsf{Exp}(\mathcal{I}^{+})$ with $[b,B] \subseteq [a,A]$ such that $[b,B] \cap \mathcal{X} = \emptyset$.
	\item[III.] A set $\mathcal{X} \subseteq [\omega]^{\omega}$ is $\mathsf{Exp}(\mathcal{I}^{+})-$\textit{meager} if $\mathcal{X}$ is the union of some countably family of $\mathsf{Exp}(\mathcal{I}^{+})-$nowhere dense sets.
	\item[IV.] A set $\mathcal{X} \subseteq [\omega]^{\omega}$ has \textit{the abstract $\mathsf{Exp}(\mathcal{I}^{+})-$Baire property} if for every $[a,A] \in \mathsf{Exp}(\mathcal{I}^{+})$ there is some $[b,B] \in \mathsf{Exp}(\mathcal{I}^{+})$ with $[b,B] \subseteq [a,A]$ such that either $[b,B] \subseteq \mathcal{X}$ or $[b,B] \cap \mathcal{X} = \emptyset$.
	\item[V.] A set $\mathcal{X} \subseteq [\omega]^{\omega}$ has \textit{the $\mathsf{Exp}(\mathcal{I}^{+})-$Baire property} if there exists some $\mathsf{Exp}(\mathcal{I}^{+})-$open set $\mathcal{O}$ such that $\mathcal{X} \triangle \mathcal{O}$ is $\mathsf{Exp}(\mathcal{I}^{+})-$meager.
\end{enumerate}

\medskip

These topological notions were originally studied by Marczewski in \cite{Marczewski}, and were later used by Farah in \cite{Farah} to characterize the local Ramsey property for subsets of $[\omega]^{\omega}$  with respect to  the family $\mathsf{Exp} (\mathcal{I}^{+})$ where $\mathcal{I}$ is a semiselective ideal on $\omega$. With those notions in mind, the following facts can be easily verified (see \cite[Definition 3.1]{Farah}):

\begin{fact}[] \label{facts_abstract_topology}
Let $\mathcal{I}$ be an ideal on $\omega$. Then, the following holds:
\setlist{nolistsep}
\begin{enumerate}
	\setlength{\itemsep}{0pt}	
    \item[(1)] The $\mathsf{Exp}(\mathcal{I}^{+})-$open sets have both the $\mathsf{Exp}(\mathcal{I}^{+})-$Baire property and the abstract $\mathsf{Exp}(\mathcal{I}^{+})-$Baire property. 
    \item[(2)] Every $\mathsf{Exp}(\mathcal{I}^{+})-$nowhere dense set is a $\mathsf{Exp}(\mathcal{I}^{+})-$meager set, and also each $\mathsf{Exp}(\mathcal{I}^{+})-$meager set has the $\mathsf{Exp}(\mathcal{I}^{+})-$Baire property.
    \item[(3)] Every $\mathsf{Exp}(\mathcal{I}^{+})-$nowhere dense set has the abstract $\mathsf{Exp}(\mathcal{I}^{+})-$Baire property, and also each set with the abstract $\mathsf{Exp}(\mathcal{I}^{+})-$Baire property has the $\mathsf{Exp}(\mathcal{I}^{+})-$Baire property.
    \item[(4)] The collection of all $\mathsf{Exp}(\mathcal{I}^{+})-$nowhere dense sets is an ideal, and also the collection of all $\mathsf{Exp}(\mathcal{I}^{+})-$meager sets is a $\sigma-$ideal
    \item[(5)] The sets with the abstract $\mathsf{Exp}(\mathcal{I}^{+})-$Baire property form an algebra, and also the sets with the $\mathsf{Exp}(\mathcal{I}^{+})-$Baire property form a $\sigma-$algebra.
\end{enumerate} 
\end{fact}

The local Ramsey property relative to an ideal $\mathcal{I}$ on $\omega$ is defined in terms of the collection $\mathsf{Exp} (\mathcal{I}^{+})$ as follows (see \cite{Farah, Matet1, Mathias, Todorcevic(BookTopology), Todorcevic(BookRamsey)}):

\begin{definition} \label{Local_Ramsey_property}
	(\textit{Local Ramsey property}). Let $\mathcal{I}$ be an ideal on $\omega$. We say that a set $\mathcal{A} \subseteq [\omega]^{\omega}$ is \textit{$\mathcal{I}^{+}-$Ramsey} if for each $s\in [\omega]^{<\omega}$ and each $C\in \mathcal{I}^{+}$ there is some $B\in \mathcal{I}^{+} \!\restriction\! C$ such that either $[s,B] \subseteq \mathcal{A}$ or $[s,B] \cap \mathcal{A} = \emptyset$. If the second option always holds, then we say that $\mathcal{A}$ is \textit{$\mathcal{I}^{+}-$Ramsey null}. 
\end{definition}

It should be noted that for every ideal $\mathcal{I}$ on $\omega$, there are subsets of $[\omega]^{\omega}$ that do not satisfy the local Ramsey property relative to $\mathcal{I}$ (see \cite[Proposition 5.2]{Matet1}).

\begin{proposition}[] \label{MM}
	(Matet, \cite{Matet1}). Let $\mathcal{I}$ be an ideal on $\omega$. For all $s\in [\omega]^{<\omega}$ and $B\in \mathcal{I}^{+}$, there exists a set $\mathcal{A} \subseteq [s,B]$ such that $\mathcal{A}$ is not $\mathcal{I}^{+}-$Ramsey.
\end{proposition}

We note the following property of $\mathcal{I}^{+}-$Ramsey null sets with respect to a semiselective ideal $\mathcal{I}$, which will be used in the proof of Proposition \ref{Additivity}.

\begin{lemma} \label{lemma_Ramsey_null}
	Let $\mathcal{I}$ be a semiselective ideal on $\omega$. If $\mathcal{C} \subseteq [\omega]^{\omega}$ is a $\mathcal{I}^{+}-$Ramsey null set, then for every $X\in \mathcal{I}^{+}$ there is some $Y \in \mathcal{I}^{+} \!\restriction\! X$ such that $\{ Z\in [\omega]^{\omega} : Z\subseteq^{*}Y \} \cap \mathcal{C} = \emptyset$. 
\end{lemma}

\begin{proof}
	Fix a semiselective ideal $\mathcal{I}$ on $\omega$, and let $\mathcal{C} \subseteq [\omega]^{\omega}$ be a $\mathcal{I}^{+}-$Ramsey null set. Given any $X\in \mathcal{I}^{+}$, we consider the sequence $\{\mathcal{D}_{n}\}_{n\in \omega}$ of dense--open sets in $(\mathcal{I}^{+} \!\restriction\! X, \subseteq)$ defined as follows:
    \begin{center}
		$\mathcal{D}_{n} = \{ N \in \mathcal{I}^{+} \!\restriction\! X : (\forall\, s<n+1)( [s,N]\cap \mathcal{C}=\emptyset ) \}$.
    \end{center}
	Since the ideal $\mathcal{I}$ is semiselective, then we can find an infinite set $Y\in \mathcal{I}^{+} \!\restriction\! X$ such that $Y/n \in \mathcal{D}_{n}$ for all $n\in Y$. Now, for every infinite set $Z \subseteq^{*} Y$ there is $n\in Y$ such that $Z/n \subseteq Y$; so, if we take the longest initial segment $s \sqsubset Z$ such that $s < n+1$, then $Z\in [s,Y/n]$ and hence $Z\notin \mathcal{C}$. Thus, we conclude that $\{ Z\in [\omega]^{\omega} : Z\subseteq^{*}Y \} \cap \mathcal{C} = \emptyset$. 
\end{proof}

\smallskip

The connection between the local Ramsey property and semiselective ideals was extensively studied by Farah and Todorcevic in \cite{Farah, Todorcevic(BookRamsey)}. The following result gives the most important characterization of semiselective ideals (see \cite[Theorem 3.1]{Farah} and \cite[Theorem 7.19]{Todorcevic(BookRamsey)}).

\begin{theorem} \label{Local_Ellentuck_Theorem}
	[\textit{Local Ellentuck Theorem}]. (Farah, \cite{Farah} $\slash$ Todorcevic, \cite{Todorcevic(BookRamsey)}). Let $\mathcal{I}$ be an ideal on $\omega$. Then, the following statements are equivalent: 
	\setlist{nolistsep}
	\begin{enumerate}
		\setlength{\itemsep}{0pt}
		\item[(A)] The ideal $\mathcal{I}$ is semiselective.
		\item[(B)] The following three families of subsets of $[\omega]^{\omega}$ form $\sigma-$ideals and coincide:
		\begin{enumerate}
			\setlength{\itemsep}{0pt}
			\item[(a)] The collection of all $\mathcal{I}^{+}-$Ramsey null sets.
			\item[(b)] The collection of all $\mathsf{Exp}(\mathcal{I}^{+})-$nowhere dense sets.
			\item[(c)] The collection of all $\mathsf{Exp}(\mathcal{I}^{+})-$meager sets.
		\end{enumerate}
		And the following three families of subsets of $[\omega]^{\omega}$ form $\sigma-$algebras and coincide:
		\begin{enumerate}
			\setlength{\itemsep}{0pt}
			\item[(a)] The collection of all $\mathcal{I}^{+}-$Ramsey sets.
			\item[(b)] The collection of all sets with the abstract $\mathsf{Exp}(\mathcal{I}^{+})-$Baire property.
			\item[(c)] The collection of all sets with the $\mathsf{Exp}(\mathcal{I}^{+})-$Baire property.
		\end{enumerate}
	\end{enumerate}
\end{theorem}

\begin{proof}	
	(A) implies (B) follows from Proposition \ref{Semiselective-Ellentuck-meager}, Proposition \ref{Semiselective-Ellentuck-Baire},  and Colollary \ref{sigma_ideal/sigma_algebra} of the next section. (B) implies (A) follows from Proposition \ref{semiselective=(p^w)+(q^+)} (see \cite[Lemma 7.5 and Lemma 7.8]{Todorcevic(BookRamsey)}).
\end{proof}

\section{A new proof of the local Ellentuck theorem}\label{newproof}

Motivated by \cite[Lemma 1]{Matet2}, we propose in this section an alternative proof of Theorem \ref{Local_Ellentuck_Theorem}, in which we do not explicitly use the local notion of combinatorial forcing applied in \cite[Section 2]{Farah} and \cite[Section 7.1]{Todorcevic(BookRamsey)}.

\begin{lemma}[] \label{Semiselective_Matet_Lemma}
	Let $\mathcal{I}$ be an ideal on $\omega$. Then, the following statements are equivalent:
	\setlist{nolistsep}	
	\begin{enumerate}
		\setlength{\itemsep}{0pt}
		\item[(a)] The ideal $\mathcal{I}$ is semiselective.
		\item[(b)] For every $a\in [\omega]^{<\omega}$ and every $A\in \mathcal{I}^{+}$, if $\{\mathcal{Q}_{i}\}_{i\in\omega} \subseteq \wp([\omega]^{\omega})$ is any sequence of subsets of $[\omega]^{\omega}$ for which $[a,B] \not\subseteq \bigcap_{i\in\omega} \mathcal{Q}_{i}$ for all $B\in \mathcal{I^{+}} \!\restriction\! A$, then there are $i\in \omega$, $c\in [\omega]^{<\omega}$, and $C\in \mathcal{I}^{+} \!\restriction\! A$, with $|a|\leq |c| \leq |a|+i$ and $[c,C] \subseteq [a,A]$, such that  $[d,D] \not\subseteq \mathcal{Q}_{i}$  for each $d\in [\omega]^{<\omega}$ and each $D\in \mathcal{I^{+}} \!\restriction\! C$ satisfying  $[d,D] \subseteq [c,C]$.	
	\end{enumerate}
\end{lemma}

\begin{proof}
	$[\text{(a)} \Longrightarrow \text{(b)}].$ Let $\mathcal{I}$ be a semiselective ideal  on $\omega$ and let $\{\mathcal{Q}_{i}\}_{i\in\omega}$ be a sequence of subsets of $[\omega]^{\omega}$ such that $[a,B] \not\subseteq \bigcap_{i\in\omega} \mathcal{Q}_{i}$ for all $B\in \mathcal{I}^{+} \!\restriction\! A$, where $a\in [\omega]^{<\omega}$ and $A\in \mathcal{I}^{+}$ are fixed. We can assume without loss of generality that $0\notin A$ and $a< \min A$.

    \medskip
	
	For each $i\in\omega$ and each $d\in[\omega]^{<\omega}$ consider the set $\mathcal{S}_{a\cup d}^{\mathcal{Q}_{i}} = \{ D\in [\omega]^{\omega} : [a\cup d, D] \subseteq \mathcal{Q}_{i} \}$, and consider also the sequence $\{\mathcal{D}_{j}\}_{j\in \omega}$ of subsets of $\mathcal{I}^{+} \!\restriction\! A$ defined as follows:
    \begin{center}
		$\mathcal{D}_{j} = \left\{ N\in \mathcal{I}^{+} \!\restriction\! A  :  (\forall \, i\leq j ) (\forall \, d\subseteq \{0,\ldots,j\}) ( \mathcal{I}^{+} \!\restriction\! N \subseteq \mathcal{S}_{a\cup d}^{\mathcal{Q}_{i}} \, \vee \, \mathcal{I}^{+} \!\restriction\! N \cap \mathcal{S}_{a\cup d}^{\mathcal{Q}_{i}} = \emptyset )  \right\}$.
    \end{center}
	
	The family $\{\mathcal{D}_{j}\}_{j\in\omega} \subseteq \wp(\mathcal{I}^{+} \!\restriction\! A)$ is a sequence of dense--open sets in $(\mathcal{I}^{+} \!\restriction\! A, \subseteq)$. Indeed, for each $j\in \omega$ we take the finite set $P_{j} = \{(i,d) \in \omega \times [\omega]^{<\omega} : i\leq j \, \wedge \, d\subseteq\{0,\ldots,j\} \}$, then we have that $\mathcal{D}_{j} = [A]^{\omega} \cap \bigcap_{(i,d)\in P_{j}} \mathcal{E}_{a\cup d}^{\mathcal{Q}_{i}}$, where $\mathcal{E}_{a\cup d}^{\mathcal{Q}_{i}}$ is the dense--open set in $(\mathcal{I}^{+}, \subseteq)$ given by $\mathcal{E}_{a\cup d}^{\mathcal{Q}_{i}} = \{ N\in\mathcal{I}^{+} : \mathcal{I}^{+} \!\restriction\! N \subseteq \mathcal{S}_{a\cup d}^{\mathcal{Q}_{i}} \, \vee \, \mathcal{I}^{+} \!\restriction\! N \cap \mathcal{S}_{a\cup d}^{\mathcal{Q}_{i}} = \emptyset \}$.

    \medskip
	
	As $\mathcal{I}$ is a semiselective ideal on $\omega$, then there exists a diagonalization $B\in \mathcal{I}^{+} \!\restriction\! A$ for the sequence $\{\mathcal{D}_{j}\}_{j\in\omega}$; thus, if $B=\{n_{k} : k\in \omega\}$ is the increasing enumeration of $B$, then $B/n_{k} \in \mathcal{D}_{n_{k}}$ for all $k\in \omega$, so that for each $i\leq n_{k}$ and each $d\subseteq \{0,\ldots,n_{k}\}$  either $\mathcal{I}^{+} \!\restriction\! (B/n_{k}) \subseteq \mathcal{S}_{a\cup d}^{\mathcal{Q}_{i}}$ or $\mathcal{I}^{+} \!\restriction\! (B/n_{k}) \cap \mathcal{S}_{a\cup d}^{\mathcal{Q}_{i}} = \emptyset$. Furthermore, to streamline the argument, it is convenient to set $n_{-1} = 0$, in which case $B/n_{-1} = B$.
    
    \medskip
    
	Since $B$ belongs to $\mathcal{I}^{+} \!\restriction\! A$, then $[a,B] \not\subseteq \bigcap_{i\in\omega} \mathcal{Q}_{i}$, so there is some $i\in \omega$ such that $[a,B] \not\subseteq \mathcal{Q}_{i}$, and hence $[a,B] \setminus \mathcal{Q}_{i} \neq \emptyset$; thus, we take $E\in [a,B] \setminus \mathcal{Q}_{i}$ and consider $c\in [\omega]^{<\omega}$ given by $c=a\cup (\{n_{j} : j<i\} \cap E)$. So, notice that $|a|\leq |c| \leq |a|+i$, and also that $a \sqsubseteq c \sqsubset E$ and $E/c \in [B/n_{i-1}]^{\omega}$, then $[c,B/n_{i-1}] \not\subseteq \mathcal{Q}_{i}$ because $E\in [c,B/n_{i-1}] \setminus \mathcal{Q}_{i}$. 

    \medskip
	
	Now, taking into account that $B/n_{i-1} \in \mathcal{I}^{+} \!\restriction\! A$, we define the sets $T_{b} \subseteq B/n_{i-1}$ and $\mathcal{K}_{b} \subseteq \mathcal{I}^{+} \!\restriction\! (B/n_{i-1})$ for every $b\in [B/n_{i-1}]^{<\omega}$ as follows:
	\begin{gather*}
		T_{b} = \{m\in B/n_{i-1} : c\cup b < m \,\wedge\, [c\cup b\cup \{m\}, B/n_{i-1}] \subseteq \mathcal{Q}_{i} \} \text{, and} 
		\\
		\mathcal{K}_{b} = \{D\in \mathcal{I}^{+} \!\restriction\! (B/n_{i-1}) : [c\cup b, D] \subseteq \mathcal{Q}_{i} \} = \mathcal{S}_{c\cup b}^{\mathcal{Q}_{i}} \cap \mathcal{I}^{+} \!\restriction\! (B/n_{i-1}).
	\end{gather*}

	\textit{Claim 1.} If $b\in [B/n_{i-1}]^{<\omega}$ and $D_{1}, D_{2} \in \mathcal{I}^{+} \!\restriction\! (B/n_{i-1})$ be such that $D_{1} \in \mathcal{K}_{b}$ and $D_{2} \subseteq D_{1}$ then $D_{2} \in \mathcal{K}_{b}$.

    \smallskip
    
	Indeed, we have $[c\cup b, D_{2}] \subseteq [c\cup b, D_{1}] \subseteq \mathcal{Q}_{i}$ whenever $D_{1} \in \mathcal{K}_{b}$ and $D_{2} \subseteq D_{1}$, hence $D_{2} \in \mathcal{K}_{b}$. 

    \medskip
	
	\textit{Claim 2.} For all $b\in [B/n_{i-1}]^{<\omega}$ it is true that either $\mathcal{K}_{b} = \mathcal{I}^{+} \!\restriction\! (B/n_{i-1})$ or $\mathcal{K}_{b}=\emptyset$. 

    \smallskip
    
	Indeed, suppose that $\mathcal{K}_{b} \neq \emptyset$ and let $D\in \mathcal{I}^{+} \!\restriction\! (B/n_{i-1})$ be such that $D\in \mathcal{K}_{b}$, hence $D/b \in \mathcal{K}_{b} \subseteq \mathcal{S}_{c\cup b}^{\mathcal{Q}_{i}}$. Now, we consider $p\in \omega$ such that $n_{p}= \min \{n_{j}\in B/n_{i-1} : c\cup b< n_{j}\}$, so that $[c\cup b, B/n_{p-1}] = [c\cup b, B/n_{i-1}]$ and $D/b=D/n_{p-1}$, hence $D/b \in \mathcal{I}^{+} \!\restriction\! (B/n_{p-1})$; moreover, since $i-1<n_{i-1}\leq n_{p-1}$ and $a \sqsubseteq c\cup b < n_{p}$, we infer that $i\leq n_{p-1}$ and $a \subseteq c\cup b \subseteq \{0,\ldots, n_{p-1}\}$. Therefore, by virtue of the fact that $B/n_{p-1} \in \mathcal{D}_{n_{p-1}}$, then either $\mathcal{I}^{+} \!\restriction\! (B/n_{p-1}) \subseteq \mathcal{S}_{c\cup b}^{\mathcal{Q}_{i}}$ or $\mathcal{I}^{+} \!\restriction\! (B/n_{p-1}) \cap \mathcal{S}_{c\cup b}^{\mathcal{Q}_{i}} = \emptyset$; however, $D/b\in \mathcal{I}^{+} \!\restriction\! (B/n_{p-1}) \cap \mathcal{S}_{c\cup b}^{\mathcal{Q}_{i}}$, so it must necessarily hold that $\mathcal{I}^{+} \!\restriction\! (B/n_{p-1}) \subseteq \mathcal{S}_{c\cup b}^{\mathcal{Q}_{i}}$, consequently $[c\cup b, B/n_{p-1}] \subseteq \mathcal{Q}_{i}$, so that $[c\cup b, B/n_{i-1}] \subseteq \mathcal{Q}_{i}$ and hence $B/n_{i-1}\in \mathcal{K}_{b}$. Thus, applying Claim 1, we deduce that $\mathcal{K}_{b} = \mathcal{I}^{+} \!\restriction\! (B/n_{i-1})$.

    \medskip
	
	\textit{Claim 3.} For all $b\in [B/n_{i-1}]^{<\omega}$ it is true that $\mathcal{K}_{b} \neq \emptyset$ if and only if $T_{b} \in \mathcal{I}^{+}$. 

    \smallskip
 
	Indeed, by Claim 2, if $\mathcal{K}_{b} \neq \emptyset$ then $\mathcal{K}_{b} = \mathcal{I}^{+} \!\restriction\! (B/n_{i-1})$, thus $B/n_{i-1} \in \mathcal{K}_{b}$ and hence $[c\cup b, B/n_{i-1}] \subseteq \mathcal{Q}_{i}$, consequently $[c\cup b\cup \{m\}, B/n_{i-1}] \subseteq [c\cup b, B/n_{i-1}] \subseteq \mathcal{Q}_{i}$ for each $m\in B/n_{i-1}$ such that $c\cup b <m$, therefore $T_{b}= \{ m\in B/n_{i-1} : c\cup b <m \}$, reason why we conclude that $T_{b} \in \mathcal{I}^{+}$. Conversely, if $T_{b}\in \mathcal{I}^{+}$ then we deduce that $[c\cup b, T_{b}] \subseteq \bigcup_{m\in T_{b}} [c\cup b \cup \{m\}, T_{b}] \subseteq \bigcup_{m\in T_{b}} [c\cup b \cup \{m\}, B/n_{i-1}] \subseteq \mathcal{Q}_{i}$, which means that $T_{b} \in \mathcal{K}_{b}$ and hence $\mathcal{K}_{b} \neq \emptyset$. 

    \medskip
	
	\textit{Claim 4.} If $b\in [B/n_{i-1}]^{<\omega}$ is such that $\mathcal{K}_{b} \neq \emptyset$ then $b\neq \emptyset$ and $\max b \in T_{b-\{\max b\}}$. 

    \smallskip
        
	Indeed, since $[c,B/n_{i-1}] \not\subseteq \mathcal{Q}_{i}$ then $B/n_{i-1} \notin \mathcal{K}_{\emptyset}$, so that $\mathcal{K}_{\emptyset} \neq \mathcal{I}^{+} \!\restriction\! (B/n_{i-1})$ and hence $\mathcal{K}_{\emptyset} = \emptyset$, as follows from Claim 2. Therefore, if $\mathcal{K}_{b} \neq \emptyset$ then $b\neq \emptyset$; moreover, by Claim 2, it follows that $\mathcal{K}_{b} = \mathcal{I}^{+} \!\restriction\! (B/n_{i-1})$, thus $B/n_{i-1} \in \mathcal{K}_{b}$ and hence $[c\cup b, B/n_{i-1}] \subseteq \mathcal{Q}_{i}$, reason why we conclude that $\max b \in T_{b-\{\max b\}}$.

    \medskip
	
	Now, let $\Gamma = \{ b\in[B/n_{i-1}]^{<\omega} : \mathcal{K}_{b} = \emptyset \}$, then $\Gamma\neq \emptyset$ because $\mathcal{K}_{\emptyset} = \emptyset$ and hence $\emptyset \in \Gamma$; so, we consider the sequence $\{\mathcal{D}_{b}^{*}\}_{b\in [\omega]^{<\omega}}$ of subsets of $\mathcal{I}^{+} \!\restriction\! (B/n_{i-1})$ given by $\mathcal{D}_{b}^{*} = \mathcal{I}^{+} \!\restriction\! (B/n_{i-1})$, except for each $b\in \Gamma$ for which we define $\mathcal{D}_{b}^{*}$ as follows:
	\begin{center}
		$\mathcal{D}_{b}^{*} = \{ N\in \mathcal{I}^{+} \!\restriction\! (B/n_{i-1}) : N \cap T_{b} = \emptyset \}$.
	\end{center}
    
	The family $\{\mathcal{D}_{b}^{*}\}_{b\in [\omega]^{<\omega}} \subseteq \wp (\mathcal{I}^{+} \!\restriction\! (B/n_{i-1}))$ defined in this way is actually a sequence of dense--open sets in $(\mathcal{I}^{+} \!\restriction\! (B/n_{i-1}), \subseteq)$. Indeed, for each $b\in \Gamma$ we have the following two facts: On the one hand, if $Y\in \mathcal{I}^{+} \!\restriction\! (B/n_{i-1})$ and $X\in \mathcal{D}_{b}^{*}$ are such that $Y\subseteq X$, then $X\cap T_{b} = \emptyset$ and hence $Y \cap T_{b} = \emptyset$, which implies that $Y\in \mathcal{D}_{b}^{*}$. Thus, $\mathcal{D}_{b}^{*}$ is an open set in $(\mathcal{I}^{+} \!\restriction\! (B/n_{i-1}), \subseteq)$. On the other hand, if $Y\in \mathcal{I}^{+} \!\restriction\! (B/n_{i-1})$ then $Y\cap T_{b} \notin \mathcal{I}^{+}$, since, by Claim 3, we get $T_{b} \notin \mathcal{I}^{+}$ as a consequence of $\mathcal{K}_{b} = \emptyset$; moreover, since $Y=(Y \setminus T_{b}) \cup (Y\cap T_{b})$ then necessarily $Y \setminus T_{b} \in \mathcal{I}^{+} \!\restriction\! (B/n_{i-1})$, and clearly the set $Y \setminus T_{b} \subseteq Y$ satisfies $(Y \setminus T_{b}) \cap T_{b} = \emptyset$, reason why we deduce that $Y \setminus T_{b} \in \mathcal{D}_{b}^{*}$. Thus, $\mathcal{D}_{b}^{*}$ is a dense set in $(\mathcal{I}^{+} \!\restriction\! (B/n_{i-1}), \subseteq)$.

    \medskip
	
	Since $\mathcal{I}$ is a semiselective ideal on $\omega$, then there exists a  diagonalization $C\in \mathcal{I}^{+} \!\restriction\! (B/n_{i-1}) \subseteq \mathcal{I}^{+} \!\restriction\! A$ for the sequence of dense--open sets $\{ \mathcal{D}_{b}^{*} \}_{b\in[\omega]^{<\omega}}$, therefore $C/b \in \mathcal{D}_{b}^{*}$ for all $b\in [C]^{<\omega}$; in addition, note that $c< \min C$ and $[c,C] \subseteq [a,A]$.

    \medskip
	
	Let us verify that $[C]^{<\omega} \subseteq \Gamma$. We use induction on the cardinality of $b$, where $b\in[C]^{<\omega}$.  Suppose that $[C]^{n} \subseteq \Gamma$, and let $b\in [C]^{n+1}$, then $b\neq \emptyset$ and $\max b \in C/(b-\{\max b\})$.  If  $b\notin \Gamma$ then $\mathcal{K}_{b} \neq \emptyset$; thus, by Claim 4, we deduce that $\max b \in T_{b-\{\max b\}}$, so that $(C/(b-\{\max b\})) \cap T_{b-\{\max b\}} \neq \emptyset$; nevertheless, since $b-\{\max b\} \in [C]^{n} \subseteq \Gamma$ and $C/(b-\{\max b\}) \in \mathcal{D}_{b-\{\max b\}}^{*}$ then $(C/(b-\{\max b\})) \cap T_{b-\{\max b\}} = \emptyset$, but this is contradictory. Therefore, we conclude that $[C]^{n+1} \subseteq \Gamma$. 

    \medskip
	
	Finally, notice that for each $d\in[\omega]^{<\omega}$ and each $D\in \mathcal{I}^{+} \!\restriction\! C$ such that $[d,D]\subseteq [c,C]$ we have that $d=c\cup b$ where $b\in[C]^{<\omega}$, then $\mathcal{K}_{b} = \emptyset$ and hence $D\notin \mathcal{K}_{b}$, which means that $[c\cup b, D] \not\subseteq \mathcal{Q}_{i}$. As a result, the set $C\in \mathcal{I}^{+} \!\restriction\! A$ satisfies that $[c,C] \subseteq [a,A]$ and it has the property that $[d,D] \not\subseteq \mathcal{Q}_{i}$ whenever $[d,D] \subseteq [c,C]$ with $D\in\mathcal{I}^{+} \!\restriction\! C$.

    \medskip
	
	$[\text{(b)} \Longrightarrow \text{(a)}].$ Suppose that the ideal $\mathcal{I}$ is not semiselective; so, by Proposition \ref{semiselective=(p^w)+(q^+)}, either $\mathcal{I}$ does not satisfy the property $(p^{w})$ or $\mathcal{I}$ does not satisfy the property $(q)$.

    \medskip
	
	Suppose first that the ideal $\mathcal{I}$ does not have the property $(p^{w})$. Let $A\in \mathcal{I}^{+}$ and let $\{\mathcal{D}_{i}\}_{i\in\omega} \subseteq \wp(\mathcal{I}^{+})$ be a sequence of dense--open sets in $(\mathcal{I}^{+}, \subseteq)$ such that for every $B\in \mathcal{I}^{+} \!\restriction\! A$ there is some $i\in\omega$ for which $B \not\subseteq^{*} D$ for each $D\in\mathcal{D}_{i}$; so, if we consider the sequence $\{Q_{i}\}_{i\in\omega} \subseteq \wp([\omega]^{\omega})$ defined by $\mathcal{Q}_{i} = \{ X\in[\omega]^{\omega} : (\exists\, D\in \mathcal{D}_{i}) (X\subseteq^{*} D)\}$, then we deduce that $[\emptyset, B] \not\subseteq \bigcap_{i\in\omega} \mathcal{Q}_{i}$ for all $B\in \mathcal{I}^{+} \!\restriction\! A$. Now, given any $i\in\omega$, $c\in [\omega]^{<\omega}$, and $C\in \mathcal{I}^{+}\!\restriction\! A$, with $|c|\leq i$ and $[c,C] \subseteq [\emptyset,A]$, we know that there exists some $D\in \mathcal{I}^{+} \!\restriction\! C$ such that $D\in \mathcal{D}_{i}$; thus, we conclude that $[c,D] \subseteq [c,C]$ and $[c,D] \subseteq \mathcal{Q}_{i}$.

    \medskip
	
	Suppose now that the ideal $\mathcal{I}$ does not have the property $(q)$. Let $A\in\mathcal{I}^{+}$ and let $A = \bigcup_{n\in\omega} f_{n}$ be an infinite partition of $A$, where each $f_{n}$ is finite, such that for every $C\in \mathcal{I}^{+} \!\restriction\! A$ there exists $J\in [\omega]^{\omega}$ for which $|f_{n}\cap C| \geq 2$ for all $n\in J$. Next, let $\{Q_{i}\}_{i\in\omega} \subseteq \wp([\omega]^{\omega})$ be the sequence defined by $\mathcal{Q}_{i} = \{ X\in[\omega]^{\omega} : (\exists\, j\geq i) (|f_{j}\cap X| \geq 2)\}$, and notice that $[\emptyset, B] \not\subseteq \bigcap_{i\in\omega} \mathcal{Q}_{i}$ for all $B\in \mathcal{I}^{+} \!\restriction\! A$. Now, given any $i\in\omega$, $c\in [\omega]^{<\omega}$, and $C\in \mathcal{I}^{+}\!\restriction\! A$, with $|c|\leq i$ and $[c,C] \subseteq [\emptyset,A]$, we can find some $j\geq i$ such that $c<\min f_{j}$ and $|f_{j}\cap C|\geq 2$; thus, if we take $d\in [\omega]^{<\omega}$ and $D\in \mathcal{I}^{+} \!\restriction\! C$ given by $d= c \cup (f_{j}\cap C)$ and $D= C/ f_{j}$, then we deduce that $[d,D] \subseteq [c,C]$ and $[d,D] \subseteq \mathcal{Q}_{i}$.	
\end{proof}

\smallskip

\begin{proposition}[] \label{Semiselective-Ellentuck-meager}
	Let $\mathcal{I}$ be a semiselective ideal on $\omega$. Then, for any set $\mathcal{A} \subseteq [\omega]^{\omega}$, the following statements are equivalent:
	\setlist{nolistsep}
	\begin{enumerate}
		\setlength{\itemsep}{0pt}	
		\item[(a)] $\mathcal{A}$ is $\mathcal{I}^{+}-$Ramsey null.
		\item[(b)] $\mathcal{A}$ is $\mathsf{Exp}(\mathcal{I}^{+})-$nowhere dense.
		\item[(c)] $\mathcal{A}$ is $\mathsf{Exp}(\mathcal{I}^{+})-$meager.
	\end{enumerate}
\end{proposition}

\begin{proof} 
	$[\text{(a)} \Longrightarrow \text{(b)} \Longrightarrow \text{(c)}].$ Straightforward.

    \smallskip
    
	$[\text{(c)} \Longrightarrow \text{(a)}].$ Fix a semiselective ideal $\mathcal{I}$ on $\omega$, and let $\mathcal{A} \subseteq [\omega]^{\omega}$ be a $\mathsf{Exp}(\mathcal{I}^{+})-$meager set, then $\mathcal{A} = \bigcup_{i\in\omega} \mathcal{N}_{i}$ where $\mathcal{N}_{i} \subseteq [\omega]^{\omega}$ is some $\mathsf{Exp}(\mathcal{I}^{+})-$nowhere dense set for each $i\in\omega$. Suppose that $\mathcal{A}$ is not a $\mathcal{I}^{+}-$Ramsey null set, then there are $a\in [\omega]^{<\omega}$ and $A\in\mathcal{I}^{+}$ such that $[a,B] \not\subseteq \mathcal{A}^{\complement}$ for every $B\in \mathcal{I}^{+} \!\restriction\! A$, so that $[a,B] \not\subseteq \bigcap_{i\in\omega} \mathcal{N}_{i}^{\complement}$ for every $B\in \mathcal{I}^{+} \!\restriction\! A$. Therefore, applying Lemma \ref{Semiselective_Matet_Lemma}, there are $i\in \omega$, $c\in[\omega]^{<\omega}$, and $C\in \mathcal{I}^{+} \!\restriction\! A$, with $[c,C]\subseteq [a,A]$, such that for each $d\in[\omega]^{<\omega}$ and each $D\in \mathcal{I}^{+} \!\restriction\! C$ we have that $[d,D] \not\subseteq \mathcal{N}_{i}^{\complement}$ whenever $[d,D]\subseteq [c,C]$, thus the set $\mathcal{N}_{i}$ is not $\mathsf{Exp}(\mathcal{I}^{+})-$nowhere dense, which is contradictory.
\end{proof}

\smallskip

\begin{proposition}[] \label{Semiselective-Ellentuck-Baire}
	Let $\mathcal{I}$ be a semiselective ideal on $\omega$. Then, for any set $\mathcal{A} \subseteq [\omega]^{\omega}$, the following statements are equivalent:
	\setlist{nolistsep}
	\begin{enumerate}
		\setlength{\itemsep}{0pt}	
		\item[(a)] $\mathcal{A}$ is $\mathcal{I}^{+}-$Ramsey.
		\item[(b)] $\mathcal{A}$ has the abstract $\mathsf{Exp}(\mathcal{I}^{+})-$Baire property.
		\item[(c)] $\mathcal{A}$ has the $\mathsf{Exp}(\mathcal{I}^{+})-$Baire property.
	\end{enumerate}
\end{proposition}

\begin{proof}
	$[\text{(a)} \Longrightarrow \text{(b)} \Longrightarrow \text{(c)}].$ Straightforward.

    \smallskip
    
	$[\text{(c)} \Longrightarrow \text{(b)}].$ This fact is easily deduced applying Proposition \ref{Semiselective-Ellentuck-meager}.

    \smallskip
    
	$[\text{(b)} \Longrightarrow \text{(a)}].$ Fix a semiselective ideal $\mathcal{I}$ on $\omega$, and let $\mathcal{A}\subseteq [\omega]^{\omega}$ be a non $\mathcal{I}^{+}-$Ramsey set, so that there are $a\in [\omega]^{<\omega}$ and $A\in \mathcal{I}^{+}$ such that both $[a,B] \not\subseteq \mathcal{A}$ and $[a,B] \not\subseteq \mathcal{A}^{\complement}$ for all $B\in \mathcal{I}^{+} \!\restriction\! A$. Firstly, let $\{\mathcal{Q}_{i}\}_{i\in\omega} \subseteq \wp([\omega]^{\omega})$ be the sequence given by $\mathcal{Q}_{0} = \mathcal{A}$ and $\mathcal{Q}_{i} = [\omega]^{\omega}$ for each $i>0$, so that $[a,B] \not\subseteq \bigcap_{i\in\omega} \mathcal{Q}_{i}$ for all $B\in \mathcal{I}^{+} \!\restriction\! A$; thus, by Lemma \ref{Semiselective_Matet_Lemma}, there is $C\in \mathcal I^{+} \!\restriction\! A$ such that $[d,D] \not\subseteq \mathcal{A}$ whenever $[d,D] \subseteq [a,C] \subseteq [a,A]$ with $D\in \mathcal{I}^{+} \!\restriction\! C$. Secondly, let $\{\mathcal{Q}_{i}^{*}\}_{i\in\omega} \subseteq \wp([\omega]^{\omega})$ be the sequence given by $\mathcal{Q}_{i}^{*} = \mathcal{A}^{\complement}$ for each $i\in\omega$, so that $[a,B] \not\subseteq \bigcap_{i\in\omega} \mathcal{Q}_{i}^{*}$ for all $B\in \mathcal{I}^{+} \!\restriction\! C$; then, by Lemma \ref{Semiselective_Matet_Lemma}, there are $e\in [\omega]^{<\omega}$ and $E\in \mathcal{I}^{+} \!\restriction\! C$, with $[e,E] \subseteq [a,C]$, such that $[d,D] \not\subseteq \mathcal{A}^{\complement}$ whenever $[d,D] \subseteq [e,E]$ with $D\in \mathcal{I}^{+} \!\restriction\! E$. Taking this into account, we have that $[e,E] \in \mathsf{Exp}(\mathcal{I}^{+})$ satisfies both $[d,D] \not\subseteq \mathcal{A}$ and $[d,D] \not\subseteq \mathcal{A}^{\complement}$ for all $[d,D]\in \mathsf{Exp}(\mathcal{I}^{+})$ with $[d,D]\subseteq [e,E]$; therefore, we conclude that the set $\mathcal{A}$ does not have the abstract $\mathsf{Exp}(\mathcal{I}^{+})-$Baire property.
\end{proof}

\smallskip

\begin{corollary} \label{sigma_ideal/sigma_algebra}
	Let $\mathcal{I}$ be a semiselective ideal on $\omega$. Then, the following hold: 
	\setlist{nolistsep}
	\begin{enumerate}
		\setlength{\itemsep}{0pt}	
		\item[(1)] The collection $\mathcal{R}_{0}(\mathcal{I})$ of all $\mathcal{I}^{+}-$Ramsey null sets forms a $\sigma-$ideal on $[\omega]^{\omega}$.
		\item[(2)] The collection $\mathcal{R}(\mathcal{I})$ of all $\mathcal{I}^{+}-$Ramsey sets forms a $\sigma-$algebra on $[\omega]^{\omega}$.
	\end{enumerate}
\end{corollary}

\begin{proof}
		This result follows from Fact \ref{facts_abstract_topology} as well as from both Proposition \ref{Semiselective-Ellentuck-meager} and Proposition \ref{Semiselective-Ellentuck-Baire}. 
\end{proof}

\section{Additivity and the generalized Suslin operation}\label{additivity}

In this section, we present some results about the additivity of the ideal of Ramsey null sets with respect to a semiselective ideal $\mathcal I$ as well as the distributivity of the quasi--order $(\mathcal I^+, \subseteq^*)$. We use then these results to characterize semiselectivity in terms of the Suslin operation $\mathcal A_\omega$.

\subsection{On additivity and distributivity}

Let $\mathcal{I}$ be an ideal on $\omega$, and let $\kappa$ be an infinite cardinal. We say that $(\mathcal{I}^{+}, \subseteq^{*})$ is \textit{$\kappa-$distributive} if the intersection of $\leq \kappa$ dense--open sets in $(\mathcal{I}^{+}, \subseteq^{*})$ is also a dense--open set in $(\mathcal{I}^{+}, \subseteq^{*})$.

\begin{definition}
	Let $\mathcal{I}$ be an ideal on $\omega$. The cardinal $\mathfrak{h}_{\mathcal{I}}$ is the smallest cardinal $\kappa$ such that $(\mathcal{I}^{+}, \subseteq^{*})$ is not $\kappa-$distributive.
\end{definition}

Naturally, the cardinal $\mathfrak{h}_{\mathcal{I}}$ is always infinite for every ideal $\mathcal{I}$ on $\omega$. Moreover, for the ideal $ [\omega]^{<\omega}$ we have $\mathfrak{h}_{[\omega]^{<\omega}} = \mathfrak{h}$, where $\mathfrak{h}$ is the cardinal characteristic of the continuum known as \textit{distributive number} or \textit{shattering number}, introduced by Balcar, Pelant, and Simon in \cite{Balcar-Pelant-Simon} (see also \cite[Chapter 9]{Halbeisen} and \cite[Section 6]{Blass}). 

\medskip

Given an ideal $\mathcal I$ on $\omega$, for every $A\in \mathcal I^+$ the set $\mathcal{D}_A =\{ B\in \mathcal{I}^+ : (|B \setminus A|<\omega \wedge |A\setminus B|=\omega) \vee  (|B\cap A|<\omega) \}$ is dense--open in $(\mathcal{I}^{+}, \subseteq^{*})$. In words, $\mathcal{D}_A$ is the collection of all $B\in \mathcal{I}^+$ strictly almost contained in $A$ or almost disjoint from $A$. The collection of dense--open sets $\{\mathcal{D}_A\}_{A\in \mathcal{I}^{+}}$ has size $\mathfrak{c}$ and its intersection is empty, since $A\notin \mathcal{D}_A$ for every $A\in \mathcal{I}^+$. This shows that $\mathfrak{h}_{\mathcal{I}} \leq \mathfrak{c}$.

\medskip

The following fact provides useful combinatorial characterizations about the distributivity of $(\mathcal{I}^{+}, \subseteq^{*})$, which can be easily verified.

\begin{fact}[] \label{p^w}
	Let $\mathcal{I}$ be an ideal on $\omega$. Then, the following statements are equivalent:
	\setlist{nolistsep}
	\begin{enumerate}
		\setlength{\itemsep}{0pt}	
		\item[(a)] $\mathfrak{h}_{\mathcal{I}} > \omega$.
		\item[(b)] $(\mathcal{I}^{+}, \subseteq^{*})$ is $\omega-$distributive.
		\item[(c)] $\mathcal{I}$ is $*-$semiselective: for each $A\in \mathcal{I}^{+}$ and for every sequence $\{ \mathcal{D}_{n} \}_{n\in \omega} \subseteq \wp(\mathcal{I}^{+})$ of dense--open sets in $(\mathcal{I}^{+}, \subseteq^{*})$, there is $D_{\infty} \in \mathcal{I}^{+}$ such that $D_{\infty} \subseteq^{*} A$ and $D_{\infty}/m \in \mathcal{D}_{m}$ for all $m\in D_{\infty}$.
		\item[(d)] $\mathcal{I}$ has the property $(p^{w})$: for every $A\in \mathcal{I}^{+}$ and every sequence $\{ \mathcal{D}_{n} \}_{n\in\omega} \subseteq \wp(\mathcal{I}^{+})$ of dense--open sets in $(\mathcal{I}^{+}, \subseteq )$, there exists some $B\in \mathcal{I}^{+} \!\restriction\! A$ with the property that for each $n\in\omega$ there is $A_{n}\in \mathcal{D}_{n}$ such that $B \subseteq^{*} A_{n}$.
		\item[(e)] $\mathcal I$ has the property $*(p^{w})$: for every $A\in \mathcal{I}^{+}$ and every sequence $\{ \mathcal{D}_{n} \}_{n\in\omega} \subseteq \wp(\mathcal{I}^{+})$ of dense--open sets in $(\mathcal{I}^{+}, \subseteq^* )$, there exists some $B\in \mathcal{I}^{+}$ with $B \subseteq^{*} A$ and with the property that for each $n\in\omega$ there is $A_{n}\in \mathcal{D}_{n}$ such that $B \subseteq^{*} A_{n}$.
	\end{enumerate}
\end{fact}

\begin{proof}
	Straightforward.
\end{proof}

\smallskip

For a given ideal $\mathcal{I}$ on $\omega$, the additivity of the ideal of Ramsey null sets relative to $\mathcal{I}$ as well as the additivity of the algebra of Ramsey sets relative to $\mathcal{I}$ are defined as follows:

\begin{definition}
	Let $\mathcal{I}$ be an ideal on $\omega$, and let $\mathcal{R}_{0}(\mathcal{I})$ be the ideal on $[\omega]^{\omega}$ formed by the $\mathcal{I}^{+}-$Ramsey null sets. The \textit{additivity of $\mathcal{R}_{0}(\mathcal{I})$} is the cardinal $\textsf{add} (\mathcal{R}_{0}(\mathcal{I})) = \min \{ |\mathcal{E}| : \mathcal{E} \subseteq \mathcal{R}_{0}(\mathcal{I}) \,\wedge\, \bigcup \mathcal{E} \notin \mathcal{R}_{0}(\mathcal{I}) \}$.
\end{definition}

\begin{definition}
	Let $\mathcal{I}$ be an ideal on $\omega$, and let $\mathcal{R}(\mathcal{I})$ be the algebra on $[\omega]^{\omega}$ formed by the $\mathcal{I}^{+}-$Ramsey sets. The \textit{additivity of $\mathcal{R}(\mathcal{I})$} is the cardinal ${\textsf{add}} (\mathcal{R}(\mathcal{I})) = \min \{ |\mathcal{E}| : \mathcal{E} \subseteq \mathcal{R}(\mathcal{I}) \,\wedge\, \bigcup \mathcal{E} \notin \mathcal{R}(\mathcal{I}) \}$.
\end{definition}

Concerning the additivity of $\mathcal{R}_{0}(\mathcal{I})$ and $\mathcal{R}(\mathcal{I})$, note that $\omega \leq \textsf{add} (\mathcal{R}_{0}(\mathcal{I})) \leq \mathfrak{c}$ and $\omega \leq {\textsf{add}} (\mathcal{R}(\mathcal{I})) \leq \mathfrak{c}$ for every  ideal $\mathcal{I}$ on $\omega$.

\medskip

The following result is due to  Matet in \cite[Proposition 9.5]{Matet1} for the class of selective ideals, extending a previous result of Plewik in \cite[Section 2]{Plewik} for the case $\mathcal{I} = [\omega]^{<\omega}$ (see also \cite[Theorem 10.2]{Halbeisen}). We improve this result, completing the argument in Matet's proof and generalizing it to any semiselective ideal $\mathcal{I}$ on $\omega$.

\begin{proposition}[] \label{Additivity}
    For every ideal $\mathcal{I}$ on $\omega$, it holds that
    \begin{center}
        $\mathfrak{h}_{\mathcal{I}} \geq \textsf{add} (\mathcal{R}_{0}(\mathcal{I})) \geq {\textsf{add}} (\mathcal{R}(\mathcal{I}))$.
    \end{center}
    Moreover, if the ideal $\mathcal{I}$ is semiselective, then
    \begin{center}
        $\mathfrak{h}_{\mathcal{I}} = \textsf{add} (\mathcal{R}_{0}(\mathcal{I})) = {\textsf{add}} (\mathcal{R}(\mathcal{I}))$.
    \end{center} 
\end{proposition}

\begin{proof}	
	First we prove that ${\textsf{add}} (\mathcal{R}(\mathcal{I})) \leq \textsf{add} (\mathcal{R}_{0}(\mathcal{I}))$ for every ideal $\mathcal{I}$. Let $\{ \mathcal{C}_{\xi} \}_{\xi\in \textsf{add} (\mathcal{R}_{0}(\mathcal{I}))} \subseteq \wp([\omega]^{\omega})$ be a collection of $\mathcal{I}^{+}-$Ramsey null sets such that its union $\mathcal C =\bigcup_{\xi\in \textsf{add} (\mathcal{R}_{0}(\mathcal{I}))} \mathcal{C}_{\xi}$ is not $\mathcal{I}^{+}-$Ramsey null. Clearly, if $\mathcal C$ is not $\mathcal{I}^{+}-$Ramsey,  we are done. Otherwise, if $\mathcal C$ is a $\mathcal{I}^{+}-$Ramsey set, then there exists some basic set $[s,B] \subseteq \mathcal{C}$ with $B\in \mathcal{I}^{+}$, and thus, there is a set  $\mathcal{A} \subseteq [s,B]$ such that $\mathcal{A}$ is not $\mathcal{I}^{+}-$Ramsey,  according to Proposition \ref{MM}; therefore,  $\mathcal{A} = \bigcup_{\xi\in \textsf{add} (\mathcal{R}_{0}(\mathcal{I}))} (\mathcal{A} \cap \mathcal{C}_{\xi})$ where each set $\mathcal{A} \cap \mathcal{C}_{\xi}$ is $\mathcal{I}^{+}-$Ramsey null.

    \medskip
	
	Now, we show that $\textsf{add} (\mathcal{R}_{0}(\mathcal{I})) \leq \mathfrak{h}_{\mathcal{I}}$ for every ideal $\mathcal{I}$. Let $\{ \mathcal{D}_{\xi} \}_{\xi\in \mathfrak{h}_{\mathcal{I}}} \subseteq \wp(\mathcal{I}^{+})$ be a collection of dense--open sets in $(\mathcal{I}^{+}, \subseteq^{*})$ such that $\bigcap_{\xi\in \mathfrak{h}_{\mathcal{I}}} \mathcal{D}_{\xi}$ is not dense in $(\mathcal{I}^{+}, \subseteq^{*})$. For every $\xi\in \mathfrak{h}_{\mathcal{I}}$ we consider the set $\mathcal{A}_{\xi} \subseteq [\omega]^{\omega}$ given by $\mathcal{A}_{\xi} = \{ X\in [\omega]^{\omega} : (\exists\, D\in \mathcal{D}_{\xi}) (X \subseteq^{*} D) \}$, and we note that each set $\mathcal{A}_{\xi}^{\complement}$ is $\mathcal{I}^{+}-$Ramsey null, because for any basic set $[s,A]$ with $A\in \mathcal{I}^{+}$, there is some $B\in \mathcal{I}^{+}\!\restriction\!  A$ such that $B\in\mathcal{D}_{\xi}$, thus $[s,B] \subseteq \mathcal{A}_{\xi}$ and hence $[s,B] \cap \mathcal{A}_{\xi}^{\complement} = \emptyset$. However, notice that the set $\bigcup_{\xi\in \mathfrak{h}_{\mathcal{I}}} \mathcal{A}_{\xi}^{\complement}$ is not $\mathcal{I}^{+}-$Ramsey null. Indeed, let $A\in \mathcal{I}^{+}$ be such that $\mathcal{I}^{+} \!\restriction\! A \cap \bigcap_{\xi\in \mathfrak{h}_{\mathcal{I}}} \mathcal{D}_{\xi} = \emptyset$. Suppose, in search of a contradiction,  that $\bigcup_{\xi\in \mathfrak{h}_{\mathcal{I}}} \mathcal{A}_{\xi}^{\complement}$ is  $\mathcal{I}^{+}-$Ramsey null, then there is some $B \in \mathcal{I}^{+} \!\restriction\! A$ such that $[B]^{\omega} \cap \bigcup_{\xi\in \mathfrak{h}_{\mathcal{I}}} \mathcal{A}_{\xi}^{\complement} = \emptyset$ and hence $[B]^{\omega} \subseteq \bigcap_{\xi\in \mathfrak{h}_{\mathcal{I}}} \mathcal{A}_{\xi}$, but this implies that $B\in \bigcap_{\xi\in \mathfrak{h}_{\mathcal{I}}} \mathcal{D}_{\xi}$, and this contradicts the choice of $A$.

    \medskip
	
	Let us now check that $\mathfrak{h}_{\mathcal{I}} \leq \textsf{add} (\mathcal{R}_{0}(\mathcal{I}))$ whenever the ideal $\mathcal{I}$ is semiselective. Let $\kappa < \mathfrak{h}_{\mathcal{I}}$ and let $\{ \mathcal{C}_{\xi} \}_{\xi \in \kappa} \subseteq \wp([\omega]^{\omega})$ be any collection of $\mathcal{I}^{+}-$Ramsey null sets. Using Lemma  \ref{lemma_Ramsey_null}, for each $\xi\in\kappa$ we can consider the dense--open set $\mathcal{D}_{\xi}$ in $(\mathcal{I}^{+}, \subseteq^{*})$ given by $\mathcal{D}_{\xi} = \{ X\in \mathcal{I}^{+} : (\forall\, Y \in [\omega]^{\omega}) (Y\subseteq^{*}X \rightarrow Y\notin \mathcal{C}_{\xi}) \}$, then the set $\bigcap_{\xi\in\kappa} \mathcal{D}_{\xi}$ is also dense--open in $(\mathcal{I}^{+}, \subseteq^{*})$ because $\kappa < \mathfrak{h}_{\mathcal{I}}$. Therefore, for every basic set $[s,A]$ with $A\in \mathcal{I}^{+}$, there is some $B\in \mathcal{I}^{+} \!\restriction\! A$ such that $B \in \bigcap_{\xi\in\kappa} \mathcal{D}_{\xi}$, thus $Y \notin \bigcup_{\xi\in\kappa} \mathcal{C}_{\xi}$ for each infinite set $Y \subseteq^{*} B$, and hence $[s,B] \cap \bigcup_{\xi\in\kappa} \mathcal{C}_{\xi} = \emptyset$. This shows that the set $\bigcup_{\xi\in\kappa} C_{\xi}$ is $\mathcal{I}^{+}-$Ramsey null, and so $\kappa < \textsf{add} (\mathcal{R}_{0}(\mathcal{I}))$.

    \medskip
	
	Finally, we prove that $\textsf{add} (\mathcal{R}_{0}(\mathcal{I})) \leq {\textsf{add}} (\mathcal{R}(\mathcal{I}))$ whenever the ideal $\mathcal{I}$ is semiselective. Let $\kappa < \textsf{add} (\mathcal{R}_{0}(\mathcal{I}))$ and let $\{ \mathcal{X}_{\xi} \}_{\xi\in\kappa} \subseteq \wp([\omega]^{\omega})$ be any collection of $\mathcal{I}^{+}-$Ramsey sets. By virtue of both Proposition \ref{Semiselective-Ellentuck-meager} and Proposition \ref{Semiselective-Ellentuck-Baire}, we deduce that for each $\xi\in\kappa$ there is some $\mathsf{Exp}(\mathcal{I}^{+})-$open set $\mathcal{O}_{\xi}$ such that $\mathcal{X}_{\xi} \triangle \mathcal{O}_{\xi}$ is a $\mathcal{I}^{+}-$Ramsey null set, and hence the set $\bigcup_{\xi\in\kappa} (\mathcal{X}_{\xi} \triangle \mathcal{O}_{\xi})$ is also $\mathcal{I}^{+}-$Ramsey null since  $\kappa < \textsf{add} (\mathcal{R}_{0}(\mathcal{I}))$; additionally, notice that $(\bigcup_{\xi\in\kappa} \mathcal{X}_{\xi}) \triangle (\bigcup_{\xi\in\kappa} \mathcal{O}_{\xi}) \subseteq \bigcup_{\xi\in\kappa} (\mathcal{X}_{\xi} \triangle \mathcal{O}_{\xi})$; thus, we conclude that the set $\bigcup_{\xi\in\kappa} \mathcal{X}_{\xi}$ is $\mathcal{I}^{+}-$Ramsey, and so  $\kappa < {\textsf{add}} (\mathcal{R}(\mathcal{I}))$.
\end{proof}

\subsection{On the generalized Suslin operation $\mathcal{A}_{\kappa}$} 

We prove that for any semiselective ideal $\mathcal I$, the $\sigma-$algebra $\mathcal R(\mathcal I)$ of all $\mathcal{I}^{+}-$Ramsey sets is closed under the generalized Suslin operation $\mathcal A_\kappa$ whenever $\omega \leq \kappa<\mathfrak{h}_{\mathcal I}$. To do this, we first present some notation and terminology that will be used in the proof of Proposition \ref{generalized_Suslin_operation}.

\medskip

Let $\kappa$ be an infinite cardinal. If $\sigma$ and $\tau$ are elements of $[\kappa]^{<\omega}$ we write $\sigma\sqsubseteq \tau$ if $\sigma$ is an initial segment of $\tau$ in their increasing enumerations. For each $\sigma\in [\kappa]^{<\omega}$, as in the case of 
$[\omega]^{\omega}$, we denote by $[\sigma]$ the set of all members of $[\kappa]^{\omega}$ that have $\sigma$ as initial segment. Also, for each $n\in\omega$ and each $Z\in[\kappa]^{\omega}$, we denote by $r_{n}(Z)$ the initial segment of $Z$ consisting of its first $n$ elements.

\medskip

Let $\mathcal{I}$ be an ideal on $\omega$ and let $\mathcal{G} = \{ \mathcal{X}_{\sigma} \}_{\sigma\in[\kappa]^{<\omega}} \subseteq \wp([\omega]^{\omega})$ be a family of $\mathcal{I}^{+}-$Ramsey sets indexed by elements of $[\kappa]^{<\omega}$, then the \textit{generalized Suslin operation $\mathcal{A}_{\kappa}$} applied to the family $\mathcal{G}$ is defined by
\begin{center}
	$\mathcal{A}_{\kappa} (\mathcal{G}) = \bigcup_{Z\in [\kappa]^{\omega}} \bigcap_{n\in\omega} \mathcal{X}_{r_{n}(Z)}$.
\end{center}

\smallskip

In \cite[Theorem 2.1]{Farah}, Farah shows that for every semiselective ideal $\mathcal{I}$, the family of $\mathcal{I}^{+}-$Ramsey sets is closed under the \textit{Suslin operation $\mathcal{A}_{\omega}$}, improving previous similar facts due to Ellentuck in \cite{Ellentuck} and Mathias in \cite{Mathias} (see also \cite[Theorem 10.9]{Halbeisen} and \cite[Theorem 9.2]{Todorcevic(BookTopology)}). Todorcevic proposes in \cite[Lemma 7.18]{Todorcevic(BookRamsey)} an alternative proof of this Farah's result. Matet shows  in \cite[Proposition 9.8]{Matet1}  that the family of $\mathcal{I}^{+}-$Ramsey sets is closed under the generalized Suslin operation $\mathcal{A}_{\kappa}$ whenever the ideal $\mathcal{I}$ is selective and $\kappa < \mathfrak{h}_{\mathcal{I}}$. Inspired on \cite[Section 7.1]{Todorcevic(BookRamsey)}, we extend the aforementioned results as follows:

\begin{proposition}[] \label{generalized_Suslin_operation}
If $\mathcal{I}$ is a semiselective ideal on $\omega$, then the collection $\mathcal{R}(\mathcal{I})$ of all $\mathcal{I}^{+}-$Ramsey sets is closed under the generalized Suslin operation $\mathcal{A}_{\kappa}$ for every infinite cardinal $\kappa<\mathfrak{h}_{\mathcal{I}}$. 
\end{proposition}

\begin{proof}
	Fix a semiselective ideal $\mathcal{I}$ on $\omega$ and an infinite cardinal $\kappa<\mathfrak{h}_{\mathcal{I}}$, and let $\mathcal{G} = \{ \mathcal{X}_{\sigma} \}_{\sigma\in[\kappa]^{<\omega}} \subseteq \wp([\omega]^{\omega})$ be a collection of $\mathcal{I}^{+}-$Ramsey sets. We aim to prove that the set $\mathcal{A}_{\kappa} (\mathcal{G}) = \bigcup_{A\in [\kappa]^{\omega}} \bigcap_{n\in\omega} \mathcal{X}_{r_{n}(A)}$ is also $\mathcal{I}^{+}-$Ramsey.

    \medskip

	For each $\sigma \in [\kappa]^{<\omega}$ we consider the set $\mathcal{X}_{\sigma}^{*} = \bigcup_{A\in [\sigma]} \bigcap_{n\in \omega} \mathcal{X}_{r_{n}(A)} \subseteq \mathcal{A}_{\kappa} (\mathcal{G})$, so,  in particular $\mathcal{X}_{\emptyset}^{*} = \mathcal{A}_{\kappa} (\mathcal{G})$. Notice that if $\tau \sqsubseteq \sigma$ then $[\sigma] \subseteq [\tau]$ and hence $\mathcal{X}_{\sigma}^{*} \subseteq \mathcal{X}_{\tau}^{*}$; on the other hand, since $\mathcal{X}_{r_{|\sigma|}(A)} = \mathcal{X}_{\sigma}$ whenever $A\in [\sigma]$, then $\mathcal{X}_{\sigma}^{*} \subseteq \mathcal{X}_{\sigma}$ for every $\sigma \in [\kappa]^{<\omega}$; additionally, since $[\sigma]= \bigcup_{\sigma<\xi} [\sigma\cup \{\xi\}]$, then for each $\sigma\in [\kappa]^{<\omega}$ it is true that 
	\begin{center}
		$\mathcal{X}_{\sigma}^{*} = \bigcup_{A\in [\sigma]} \bigcap_{n\in \omega} \mathcal{X}_{r_{n}(A)} = \bigcup_{\sigma<\xi} \bigcup_{A\in [\sigma\cup \{\xi\}]} \bigcap_{n\in \omega} \mathcal{X}_{r_{n}(A)} = \bigcup_{\sigma<\xi} \mathcal{X}_{\sigma\cup\{\xi\}}^{*}$.
	\end{center}

    \smallskip
    
	Fix arbitrary $a\in [\omega]^{<\omega}$ and $M\in \mathcal{I}^{+}$. We begin by considering the sequence $\{ \mathcal{D}_{z} \}_{z\in [\omega]^{<\omega}}$ of subsets of $\mathcal{I}^{+} \!\restriction\! M$ given by $\mathcal{D}_{z} = \mathcal{I}^{+} \!\restriction\! M$, except for each $z\in [\omega/a]^{<\omega}$ with $z\neq \emptyset$ where we define $\mathcal{D}_{z}$ as follows:
	\begin{equation*}
		\mathcal{D}_{z} = \left\{ N \in \mathcal{I}^{+} \!\restriction\! M \, : \, \forall \, \sigma\in [\kappa]^{< |z|}
		\begin{pmatrix}
			[a \cup z, N] \subseteq \mathcal{X}_{\sigma}^{*} \; \; \vee \; \; [a \cup z, N] \cap \mathcal{X}_{\sigma}^{*} = \emptyset \; \; \; \vee \\ (\forall \, P\in \mathcal{I}^{+} \!\restriction\! N) ( \, [a \cup z, P] \not\subseteq \mathcal{X}_{\sigma}^{*} \, \wedge \, [a \cup z, P] \cap \mathcal{X}_{\sigma}^{*} \neq \emptyset \, )
		\end{pmatrix}
		 \right\}.
	\end{equation*}
	
	It is easy to check that the family $\{ \mathcal{D}_{z} \}_{z\in[\omega]^{<\omega}} \subseteq \wp(\mathcal{I}^{+} \!\restriction\! M)$ defined in this way is a sequence of dense--open sets in $(\mathcal{I}^{+} \!\restriction\! M, \subseteq)$.

    \medskip
	
	Since $\mathcal{I}$ is a semiselective ideal on $\omega$, there exists some set $D_{\infty} \in \mathcal{I}^{+} \!\restriction\! M$ such that $D_{\infty}/b \in \mathcal{D}_{b}$ for all $b\in [D_{\infty}]^{<\omega}$. For each $\sigma\in [\kappa]^{<\omega}$, let  $\mathcal{F}_{\sigma} \subseteq [\omega]^{<\omega}$ be the set described by $\mathcal{F}_{\sigma} = \{ b\in [D_{\infty}/a]^{<\omega} : \sigma \in [\kappa]^{<|b|} \wedge [a\cup b, D_{\infty}] \cap \mathcal{X}_{\sigma}^{*} = \emptyset \}$; so, for each $\sigma\in [\kappa]^{<\omega}$ we now consider the set $\Phi_{\sigma} \subseteq [\omega]^{\omega}$ defined by $\Phi_{\sigma} = [a,D_{\infty}] \setminus \bigcup_{b\in \mathcal{F}_{\sigma}} [a\cup b, D_{\infty}]$, which visibly satisfies that $[a,D_{\infty}] \cap \mathcal{X}_{\sigma}^{*} \subseteq \Phi_{\sigma} \subseteq [a,D_{\infty}]$.

    \medskip
	
	Notice that for every $\sigma\in [\kappa]^{<\omega}$ the set $\Phi_{\sigma}$ has the $\mathsf{Exp}(\mathcal{I}^{+})-$Baire property, by virtue of Fact \ref{facts_abstract_topology}; thus, applying Proposition \ref{Semiselective-Ellentuck-Baire}, it follows that $\Phi_{\sigma}$ is a $\mathcal{I}^{+}-$Ramsey set. Now, for each $\sigma\in [\kappa]^{<\omega}$ we consider the $\mathcal{I}^{+}-$Ramsey set $\Psi_{\sigma} \subseteq [\omega]^{\omega}$ given by $\Psi_{\sigma} = \mathcal{X}_{\sigma} \cap \Phi_{\sigma}$, which clearly satisfies that $[a,D_{\infty}] \cap \mathcal{X}_{\sigma}^{*} \subseteq \Psi_{\sigma} \subseteq \Phi_{\sigma} \subseteq [a,D_{\infty}]$.

    \medskip
	
	Next, for each $\sigma\in [\kappa]^{<\omega}$ we define the set $\mathcal{M}_{\sigma} \subseteq [\omega]^{\omega}$ by $\mathcal{M}_{\sigma} = \Psi_{\sigma} \setminus \bigcup_{\sigma <\xi} \Psi_{\sigma\cup \{\xi\}}$, which is a $\mathcal{I}^{+}-$Ramsey set according to Proposition \ref{Additivity} and Corollary \ref{sigma_ideal/sigma_algebra}; thus, the set $\mathcal{M}_{\sigma}$ has the abstract $\mathsf{Exp}(\mathcal{I}^{+})-$Baire property by virtue of Proposition \ref{Semiselective-Ellentuck-Baire}. Furthermore, since $\mathcal{M}_{\sigma} \subseteq [a,D_{\infty}]$ and $[a,D_{\infty}] \cap \mathcal{X}_{\sigma}^{*} = [a,D_{\infty}] \cap \bigcup_{\sigma <\xi} \mathcal{X}_{\sigma\cup\{\xi\}}^{*} \subseteq \bigcup_{\sigma <\xi} \Psi_{\sigma\cup\{\xi\}}$, then it follows that $\mathcal{M}_{\sigma} \cap \mathcal{X}_{\sigma}^{*} = \emptyset$, and hence we conclude that $\mathcal{M}_{\sigma} \subseteq \Psi_{\sigma} \setminus \mathcal{X}_{\sigma}^{*} \subseteq \Phi_{\sigma} \setminus \mathcal{X}_{\sigma}^{*}$. 

    \medskip
	
	Let us check now that $\Psi_{\emptyset} \setminus \mathcal{A}_{\kappa} (\mathcal{G}) \subseteq \bigcup_{\sigma\in [\kappa]^{<\omega}} \mathcal{M}_{\sigma}$, which is naturally equivalent to verifying that $\Psi_{\emptyset} \setminus \bigcup_{\sigma\in [\kappa]^{<\omega}} \mathcal{M}_{\sigma} \subseteq \mathcal{A}_{\kappa} (\mathcal{G})$. Indeed, if $E\in \Psi_{\emptyset} \setminus \bigcup_{\sigma\in [\kappa]^{<\omega}} \mathcal{M}_{\sigma}$ then we can recursively construct a set $A\in [\kappa]^{\omega}$ such that $E\in \Psi_{r_{n}(A)} \subseteq \mathcal{X}_{r_{n}(A)}$ for every $n\in \omega$, thus $E\in \bigcup_{A\in [\kappa]^{\omega}} \bigcap_{n\in\omega} \mathcal{X}_{r_{n}(A)}$ and hence $E\in \mathcal{A}_{\kappa} (\mathcal{G})$. 

    \medskip
	
	Let us verify that for each $\sigma\in [\kappa]^{<\omega}$ it is true that the set with the abstract $\mathsf{Exp}(\mathcal{I}^{+})-$Baire property $\mathcal{M}_{\sigma}$ is itself a $\mathsf{Exp}(\mathcal{I}^{+})-$nowhere dense set. Indeed, in order to reach a contradiction, suppose that for some $\sigma_{0}\in [\kappa]^{<\omega}$ the set $\mathcal{M}_{\sigma_{0}} \subseteq [a,D_{\infty}]$ is not $\mathsf{Exp}(\mathcal{I}^{+})-$nowhere dense, then there are $c\in [\omega]^{<\omega}$ and $Q\in [a,D_{\infty}] \cap \mathcal{I}^{+}$ such that $[c,Q] \subseteq \mathcal{M}_{\sigma_{0}}$ and hence $[c,Q] \subseteq [a,D_{\infty}]$, so that $a \sqsubseteq c$, and for $b = c/a = \{ n\in c : a<n \}$ we infer that $c = a \cup b$ and $b \in [D_{\infty}/a]^{<\omega}$; in fact, we can assume without loss of generality that $\sigma_{0} \in [\kappa]^{<|b|}$. Now, since $[a \cup b, Q] \subseteq \mathcal{M}_{\sigma_{0}} \subseteq \Phi_{\sigma_{0}} \setminus \mathcal{X}_{\sigma_{0}}^{*}$, then clearly we have that $[a \cup b, Q] \subseteq \Phi_{\sigma_{0}}$ and $[a \cup b, Q] \cap \mathcal{X}_{\sigma_{0}}^{*} = \emptyset$; on the other hand, since $[a \cup b, D_{\infty}/b] = [a \cup b, D_{\infty}]$ and $D_{\infty}/b \in \mathcal{D}_{b}$, then either $[a \cup b, D_{\infty}] \subseteq \mathcal{X}_{\sigma_{0}}^{*}$, or $[a \cup b, D_{\infty}] \cap \mathcal{X}_{\sigma_{0}}^{*} = \emptyset$, or $[a \cup b, P] \not\subseteq \mathcal{X}_{\sigma_{0}}^{*}$ and $[a \cup b, P] \cap \mathcal{X}_{\sigma_{0}}^{*} \neq \emptyset$ for all $P \in \mathcal{I}^{+} \!\restriction\! (D_{\infty}/b)$, however both the first and the last of these three options are impossible, because $Q\in [a,D_{\infty}] \cap \mathcal{I}^{+}$ satisfies that $[a \cup b,Q/b] = [a \cup b, Q] \subseteq [a \cup b, D_{\infty}]$ and $[a \cup b, Q] \cap \mathcal{X}_{\sigma_{0}}^{*} = \emptyset$. Therefore, we deduce that $[a \cup b, D_{\infty}] \cap \mathcal{X}_{\sigma_{0}}^{*} = \emptyset$ and consequently $b\in \mathcal{F}_{\sigma_{0}}$, thus $[a \cup b, D_{\infty}] \cap \Phi_{\sigma_{0}} = \emptyset$ and hence $[a \cup b, Q] \cap \Phi_{\sigma_{0}} = \emptyset$, which is contradictory.

    \medskip
	
	Finally, since the set $\mathcal{M}_{\sigma}$ is $\mathsf{Exp}(\mathcal{I}^{+})-$nowhere dense for each $\sigma\in [\kappa]^{<\omega}$, then by Proposition \ref{Semiselective-Ellentuck-meager} we have that for every $\sigma\in [\kappa]^{<\omega}$ the set $\mathcal{M}_{\sigma}$ is itself $\mathcal{I}^{+}-$Ramsey null, so that $\bigcup_{\sigma\in [\kappa]^{<\omega}} \mathcal{M}_{\sigma}$ is also a $\mathcal{I}^{+}-$Ramsey null set, according to Proposition \ref{Additivity}. Therefore, there is some $S\in \mathcal{I}^{+} \!\restriction\! D_{\infty}$ such that $[a,S] \cap \bigcup_{\sigma\in [\kappa]^{<\omega}} \mathcal{M}_{\sigma} = \emptyset$, so that $[a,S] \cap ( \Psi_{\emptyset} \setminus \mathcal{A}_{\kappa} (\mathcal{G}) ) = \emptyset$ due to $\Psi_{\emptyset} \setminus \mathcal{A}_{k} (\mathcal{G}) \subseteq \bigcup_{\sigma\in [\kappa]^{<\omega}} \mathcal{M}_{\sigma}$; moreover, considering that $[a,D_{\infty}] \cap \mathcal{A}_{\kappa} (\mathcal{G}) \subseteq \Psi_{\emptyset} \subseteq [a,D_{\infty}]$, we deduce that $[a,S] \cap \Psi_{\emptyset} = [a,S] \cap \mathcal{A}_{\kappa} (\mathcal{G})$. Thus, since the set $\Psi_{\emptyset}$ is $\mathcal{I}^{+}-$Ramsey, then there exists some $R \in \mathcal{I}^{+} \!\restriction\! S \subseteq \mathcal{I}^{+} \!\restriction\! D_{\infty}$ such that either $[a,R] \subseteq \Psi_{\emptyset}$ or $[a,R] \cap \Psi_{\emptyset} = \emptyset$, which implies that either $[a,R] \subseteq \mathcal{A}_{\kappa} (\mathcal{G})$ or $[a,R] \cap \mathcal{A}_{\kappa} (\mathcal{G}) = \emptyset$. As a result, we conclude that the set $\mathcal{A}_{\kappa} (\mathcal{G})$ is $\mathcal{I}^{+}-$Ramsey. 
\end{proof}

\smallskip

We now present a characterization of semiselective ideals through the Suslin operation $\mathcal{A}_{\omega}$ on the algebra of all sets with the local Ramsey property.

\begin{proposition}[] 
	Let $\mathcal{I}$ be an ideal on $\omega$. Then, the following statements are equivalent:
	\setlist{nolistsep}
	\begin{enumerate}
		\setlength{\itemsep}{0pt}	
		\item[(a)] The ideal $\mathcal{I}$ is semiselective.
		\item[(b)] The collection $\mathcal{R}(\mathcal{I})$ of all $\mathcal{I}^{+}-$Ramsey sets is closed under the Suslin operation $\mathcal{A}_{\omega}$.
	\end{enumerate}
\end{proposition}

\begin{proof}
	$[\text{(a)} \Longrightarrow \text{(b)}].$ If the ideal $\mathcal{I}$ is semiselective, then $\mathfrak{h}_{\mathcal{I}} > \omega$ by virtue of Proposition \ref{semiselective=(p^w)+(q^+)} and Fact \ref{p^w}, hence the collection $\mathcal{R}(\mathcal{I})$ of all $\mathcal{I}^{+}-$Ramsey sets is closed under the Suslin operation $\mathcal{A}_{\omega}$, according to Proposition \ref{generalized_Suslin_operation}.

    \medskip
	
	$[\text{(b)} \Longrightarrow \text{(a)}].$ If the ideal $\mathcal{I}$ is not semiselective, then either $\mathcal{I}$ fails to satisfy the property $(p^{w})$ or $\mathcal{I}$ fails to satisfy the property $(q)$, according to Proposition \ref{semiselective=(p^w)+(q^+)}.

    \medskip
	
	On the one hand, suppose that the ideal $\mathcal{I}$ does not have the property $(p^{w})$. By Fact \ref{p^w}, let $A\in \mathcal{I}^{+}$ and let $\{ \mathcal{D}_{n} \}_{n\in\omega} \subseteq \wp (\mathcal{I}^{+})$ be a sequence of dense--open sets in $(\mathcal{I}^{+}, \subseteq^{*})$ such that for every $E\in \mathcal{I}^{+}$, with $E\subseteq^{*} A$, there is some $n\in E$ for which $E/n \notin \mathcal{D}_{n}$. Now, for every $n\in\omega$ we consider the set $\mathcal{C}_{n} \subseteq [\omega]^{\omega}$ given by $\mathcal{C}_{n} = \{ X\in [\omega]^{\omega} : X/n \notin \mathcal{D}_{n}^{*} \}$, where $\mathcal{D}_{n}^{*} = \{ X\in [\omega]^{\omega} : (\exists\, D\in \mathcal{D}_{n}) (X \subseteq^{*} D) \}$, so that $\mathcal{D}_{n} = \mathcal{I}^{+} \cap \mathcal{D}_{n}^{*}$. 

    \medskip
	
	Notice that $\{ \mathcal{C}_{n} \}_{n\in\omega} \subseteq \wp([\omega]^{\omega})$ is a collection of $\mathcal{I}^{+}-$Ramsey null sets, because for every $n\in\omega$ and every basic set $[s,M]$ with $M\in\mathcal{I}^{+}$, there exists $N\in \mathcal{I}^{+} \!\restriction\! M$ such that $N \in \mathcal{D}_{n}$ and hence $\{ X\in [\omega]^{\omega} : X\subseteq^{*} N \} \subseteq \mathcal{D}_{n}^{*}$, in particular $X/n \in \mathcal{D}_{n}^{*}$ for all infinite set $X\subseteq^{*} N$, therefore $[s,N] \cap \mathcal{C}_{n} = \emptyset$. Nevertheless, the set $\bigcup_{n\in\omega} \mathcal{C}_{n}$ is not $\mathcal{I}^{+}-$Ramsey, since for every $B \in \mathcal{I}^{+} \!\restriction\! A$ there is $n\in B$ such that $B/n \notin \mathcal{D}_{n}$, in fact $B/n \notin \mathcal{D}_{n}^{*}$, so that $B\in \mathcal{C}_{n}$ and hence $[B]^{\omega} \cap \bigcup_{n\in\omega} \mathcal{C}_{n} \neq \emptyset$; moreover, for any sequence $\{B_{n}\}_{n\in\omega} \subseteq \mathcal{I}^{+} \!\restriction\! B$ satisfying $B_{n}\in \mathcal{D}_{n}$ and $B_{n+1} \subseteq^{*} B_{n}$ for each $n\in\omega$, there is $E\in [B]^{\omega}$ such that for each $n\in\omega$ we have $E/n \subseteq^{*} B_{n}$ and hence $E/n \in \mathcal{D}_{n}^{*}$, which implies that $E\in \bigcap_{n\in\omega} \mathcal{C}_{n}^{\complement}$, and thus $[B]^{\omega} \cap (\bigcup_{n\in\omega} \mathcal{C}_{n})^{\complement} \neq \emptyset$.
    
    \medskip
	
	Lastly, we consider the family of $\mathcal{I}^{+}-$Ramsey sets $\mathcal{G} = \{ \mathcal{X}_{s} \}_{s\in [\omega]^{<\omega}} \subseteq \wp([\omega]^{\omega})$ defined by:
	\begin{center}
		$\mathcal{X}_{s} = \mathcal{C}_{n}^{\complement}$ whenever $s\in [\omega]^{n}$.
	\end{center}
	Therefore, applying the Suslin operation $\mathcal{A}_{\omega}$ to the collection $\mathcal{G}$, we deduce that $\mathcal{A}_{\omega} (\mathcal{G}) = \bigcup_{Z\in [\omega]^{\omega}} \bigcap_{n\in\omega} \mathcal{X}_{r_{n}(Z)} = ( \bigcup_{n\in\omega} \mathcal{C}_{n} )^{\complement}$ is not a $\mathcal{I}^{+}-$Ramsey set.

    \medskip
	
	On the other hand, suppose that the ideal $\mathcal{I}$ does not have the property $(q)$. Let $A\in\mathcal{I}^{+}$ and let $A = \bigcup_{n\in\omega} f_{n}$ be an infinite partition of $A$, where each $f_{n}$ is finite, such that for every $B\in \mathcal{I}^{+} \!\restriction\! A$ there is some $n\in\omega$ such that $f_{n}\cap B$ has at least two elements. 

    \medskip
	
	Next, we consider the set $\mathcal{Q} \subseteq [\omega]^{\omega}$ given by $\mathcal{Q} = \{ M\in [\omega]^{\omega} : (\exists\, n\in\omega)( \{m_{0},m_{1}\} \subseteq f_{n} ) \}$, where $m_{0} = \min M$ and $m_{1} = \min (M/ \min M)$. Now, notice that the set $\mathcal{Q}$ is not $\mathcal{I}^{+}-$Ramsey, since for every $B\in \mathcal{I}^{+} \!\restriction\! A$ there are $n\in\omega$ and $a,a^{\prime} \in A$ such that $a \neq a^{\prime}$ and $\{a,a^{\prime}\} \subseteq f_{n}\cap B$; thus, if $\ell = \max (f_{n}\cap B)$ then $\{a,a^{\prime}\} \cup (B/\ell) \in \mathcal{Q}$ and $\{\ell\}\cup(B/\ell) \notin \mathcal{Q}$, hence we have that $[B]^{\omega} \cap \mathcal{Q} \neq \emptyset$ and $[B]^{\omega} \cap \mathcal{Q}^{\complement} \neq \emptyset$.

    \medskip
	
	Finally, we consider the family of $\mathcal{I}^{+}-$Ramsey sets $\mathcal{G} = \{ \mathcal{X}_{s} \}_{s\in [\omega]^{<\omega}} \subseteq \wp([\omega]^{\omega})$ defined by:
	\begin{center}
		$\mathcal{X}_{s} =
		\begin{cases}
			[\omega]^{\omega} & \text{ if } s \notin [\omega]^{2} \\  
			[s] & \text{ if } s \in [\omega]^{2} \text{ and } s\in \bigcup_{n\in\omega} [f_{n}]^{2} \\
			\emptyset & \text{ if } s \in [\omega]^{2} \text{ and } s\notin \bigcup_{n\in\omega} [f_{n}]^{2}
		\end{cases}$.
	\end{center}
	Therefore, applying the Suslin operation $\mathcal{A}_{\omega}$ to the collection $\mathcal{G}$, we conclude that $\mathcal{A}_{\omega} (\mathcal{G}) = \bigcup_{Z\in [\omega]^{\omega}} \bigcap_{n\in\omega} \mathcal{X}_{r_{n}(Z)} = \mathcal{Q}$ is not a $\mathcal{I}^{+}-$Ramsey set. 	
\end{proof}

\section{A semiselective coideal containing no selective ultrafilter}\label{chandSuslin}

Denote by $\mathrm{CH}$ the continuum hypothesis, and by $\neg \mathrm{SH}$ the negation of the Suslin hypothesis; thus, $\mathrm{CH}$ is the statement ``$\mathfrak{c} = \aleph_{1}$'', and $\neg \mathrm{SH}$ is the statement ``there is a Suslin tree''. With respect to the concepts and techniques used in this section, we suggest the reader to see, for example, \cite{Bartoszynski-Judah, Halbeisen, Jech(Book), Kunen(Book)}.

\medskip

At this point, it is worth recalling that a non--principal ultrafilter $\mathcal{U}$ on $\omega$ is \textit{selective} if it is a maximal selective coideal. Equivalently, $\mathcal{U}$ is selective if for every infinite partition $\omega = \bigcup_{n\in\omega} P_{n}$, with each $P_{n} \notin \mathcal{U}$, there is $X\in \mathcal{U}$ such that $|X \cap P_{n}| \leq 1$ for each $n\in \omega$. 

\medskip

Similarly, it is useful to recall that a non--principal ultrafilter $\mathcal U$ on $\omega$ is a \textit{p--point} if it is a maximal coideal satisfying property $(p)$, that is, $\mathcal U$ is a p--point if for every countable descending chain $\{A_n\}_{n\in \omega}$ of elements of $\mathcal U$ there is $A \in \mathcal U$ such that $A\subseteq^* A_n$ for each $n\in \omega$. In particular, every selective ultrafilter is a p--point.

\medskip

While any decreasing sequence of elements of a selective coideal has a diagonalization in the coideal, in the case of a semiselective coideal $\mathcal{I}^{+}$, we only have that for any sequence $\{\mathcal{D}_{n}\}_{n\in\omega}$ of dense--open sets in $(\mathcal{I}^{+}, \subseteq)$, there is at least one decreasing sequence $\{A_{n}\}_{n\in\omega}$ of elements of $\mathcal{I}^{+}$, such that $A_{n}\in \mathcal{D}_{n}$ for every $n\in \omega$, which has a diagonalization in the coideal $\mathcal{I}^{+}$. This difference is important when investigating under which conditions semiselective coideals contain selective ultrafilters. 

\medskip

In \cite[Proposition 0.11]{Mathias}, Mathias proved that $\mathrm{CH}$ implies that every selective coideal contains a selective ultrafilter, and we include a proof of this fact to illustrate the difference between selectivity and semiselectivity for ideals.

\begin{proposition}[]
(Mathias, \cite{Mathias}). Suppose $\textrm{CH}$ holds. Then, every selective coideal on $\omega$ contains a selective ultrafilter.
\end{proposition}

\begin{proof}
Let $\mathcal{I}$ be a selective ideal on $\omega$, so that $\mathcal{I}$ satisfies the properties $(p)$ and $(q)$ stated in Proposition \ref{selective=(p)+(q)}. Using $\textrm{CH}$, list all partitions of $\omega$ as $\{\mathcal A_\alpha : \alpha \in\omega_1\}$, and for each $\alpha\in\omega_1$, let $\mathcal{A}_{\alpha} = \{ A^{\alpha}_{n} : n\in d_{\alpha} \}$ where $d_{\alpha}\leq \omega$.

\medskip

We recursively construct an uncountable decreasing sequence $\{X_\alpha\}_{\alpha\in\omega_1}$ on $(\mathcal I^+, \subseteq^*)$, putting initially $X_0=\omega$. Now, let $\alpha \in \omega_1$ and suppose that we have defined $X_\alpha \in \mathcal{I}^{+}$. If $X_\alpha\cap A^\alpha_n\in \mathcal I^+$ for some $n\in d_{\alpha}$, put $X_{\alpha+1}= X_\alpha\cap A^\alpha_n$. Otherwise, $\mathcal A_{\alpha}$ must be a partition into infinitely many pieces and $X_\alpha\cap A^\alpha_n\in \mathcal{I}$ for all  $n\in \omega$. Consequently, we have that $\{ X_{\alpha} \setminus \bigcup_{k=0}^{n} A^{\alpha}_{k} \}_{n\in\omega}$ is a decreasing sequence of elements of $\mathcal{I}^{+} \!\restriction\! X_{\alpha}$, then by property $(p)$, there exists some $B\in \mathcal I^+ \!\restriction\! X_{\alpha}$ such that $B \subseteq^{*} X_{\alpha} \setminus \bigcup_{k=0}^{n} A^{\alpha}_{k}$ for each $n\in\omega$; thus, $B \cap A^{\alpha}_{n}$ is a finite set for every $n\in\omega$. By property $(q)$, let $X_{\alpha+1}\in \mathcal I^+ \!\restriction\! B \subseteq  \mathcal I^+ \!\restriction\! X_{\alpha}$ be such that $|X_{\alpha+1}\cap A^\alpha_n|\leq 1$ for each $n\in\omega$. Finally, let $\lambda \in \omega_1 $ be a countable limit ordinal and suppose that the decreasing sequence $\{X_\alpha\}_{\alpha\in\lambda}$ on $(\mathcal I^+, \subseteq^*)$ has been defined. By property $(p)$, let $X_\lambda \in \mathcal{I}^{+}$ be such that $X_\lambda \subseteq^* X_\alpha$ for all $\alpha\in\lambda$.

\medskip

Next, let $\mathcal{U}$ be the filter on $\omega$ generated by the sequence $\{ X_{\alpha} \}_{\alpha\in \omega_{1}} \subseteq \mathcal{I}^{+}$, so that $\mathcal{U} = \{X\in [\omega]^{\omega} : (\exists\, \alpha\in \omega_1) ( X_\alpha \subseteq^{*} X)\}$ and hence $\mathcal{U} \subseteq \mathcal{I}^{+}$. Notice that $\mathcal{U}$ is an ultrafilter, since for every $A\subseteq \omega$ the pair $\{A, \omega\setminus A\}$ is a partition $\mathcal{A}_{\alpha}$ of $\omega$ that was considered at some stage $\alpha\in\omega_{1}$ of the construction, then either $X_{\alpha+1} \subseteq A$ or $X_{\alpha+1} \subseteq \omega\setminus A$. Additionally, the ultrafilter $\mathcal U$ is selective, since by construction every partition of $\omega$ into infinitely many pieces outside $\mathcal{U}$ is of the form $\mathcal A_\alpha$ for some $\alpha\in\omega_{1}$, thus $X_{\alpha} \cap A^\alpha_n \in \mathcal{I}$ for all $n\in\omega$, and hence $X_{\alpha+1}\in \mathcal U$ satisfies that $|X_{\alpha+1}\cap A^\alpha_n|\leq 1$ for each $n\in \omega$.
\end{proof}

\smallskip

Let us remember some concepts related to trees (see, for example, \cite[Chapter 3]{Kunen(Book)}). An  $\omega_1-$tree is a tree of height $\omega_1$ such that all of its levels are countable. A \textit{Suslin tree} is an $\omega_1-$tree in which all chains and antichains are countable. A tree is rooted if it has a single element in its first level. Suslin's Hypothesis $\mathrm{SH}$ can be formulated as  ``there are no Suslin trees''.  So, $\neg \mathrm{SH}$ stands for ``there is a Suslin tree''.  A Suslin tree $T$ is well pruned if it is rooted and for every $t\in T$ the collection of all elements of the tree above $t$ is uncountable. The existence of well pruned Suslin trees follows from $\neg \mathrm{SH}$; thus, if there exists a Suslin tree then there is a well pruned one. 

\medskip

We will now show that, under $\mathrm{CH}$ and $\neg \mathrm{SH}$, there is a semiselective coideal that does not contain a p--point, and hence it does not contain a selective ultrafilter. Therefore, in $L$ not every semiselective coideal contains a selective ultrafilter.

\begin{theorem}\label{seminoselultra}
Suppose $\mathrm{CH}$ holds together with $\neg \mathrm{SH}$. Then, there exists a semiselective coideal on $\omega$ that does not contain a p--point. 
\end{theorem}

\begin{proof}
Assume $\mathrm{CH}$ and let $(\mathcal{S}, \leq)$ be a well pruned Suslin tree. First of all, suppose there exists an order embedding
\begin{center}
	$\varphi: (\mathcal{S}, \leq) \longrightarrow  ([\omega]^\omega, ^*\!\supseteq)$
\end{center}  
satisfying the following conditions:
\setlist{nolistsep}
\begin{enumerate}
	\setlength{\itemsep}{0pt}
	\item[(1)] For every $s\in \mathcal S$ and every finite partition $\varphi(s) = X_0\cup X_1 $, there is $t\in \mathcal S$ such that $t\geq s$ and $\varphi(t) \subseteq^* X_i$ for some $i\in\{0,1\}$. 
	\item[(2)] For every $s\in \mathcal S$ and every infinite partition $\varphi(s) = \bigcup_{n\in \omega} f_n$, with each $f_n$ finite, there is $t\in \mathcal S$ such that $t\geq s $ and $|\varphi(t) \cap f_n|\leq 1$ for each $n\in \omega$.
\end{enumerate}

\medskip

It is important to specify that by an \textit{order embedding} $\varphi: (\mathcal{S}, \leq) \rightarrow ([\omega]^\omega, ^*\!\supseteq)$ we mean that $\varphi$ is an injective function satisfying that $s\leq t$ if and only if $\varphi(t) \subseteq^{*} \varphi(s)$ for all $s,t \in \mathcal{S}$.

\medskip

Consider the collection $\mathcal{I} =\{ A\subseteq \omega : (\forall\, s\in \mathcal{S}) (\varphi(s) \not\subseteq^{*} A) \}$. We show that its complement
\begin{center}
	$\mathcal{I}^{+} = \{ A \subseteq \omega : (\exists\, s \in \mathcal{S}) (\varphi(s)\subseteq^{*} A) \}$
\end{center}  
is a semiselective coideal on $\omega$ such that it does not contain a p--point.

\medskip

It is easy to verify that $\mathcal I^+$ is a coideal on $\omega$. Clearly, $\mathcal I^+$ is a non--empty family of infinite subsets of $\omega$ such that it is closed under taking supersets. Moreover, by condition $(1)$ of $\varphi$, we deduce that if $A\cup B \in \mathcal{I}^{+}$ then either $A\in \mathcal{I}^{+}$ or $B\in\mathcal{I}^{+}$. 

\medskip

Let us check that the coideal $\mathcal I^+$ is semiselective using Proposition \ref{semiselective=(p^w)+(q^+)}. Indeed, by condition $(2)$ of $\varphi$, clearly $\mathcal I^+$ satisfies property $(q)$. To show that it also satisfies property $(p^w)$, or equivalently property $*(p^w)$ according to Fact \ref{p^w}, notice that the image under $\varphi$ of a maximal antichain of the tree $\mathcal S$ is also a maximal antichain in $\varphi[\mathcal S]$, the image of $\mathcal S$ under $\varphi$. If $\mathcal A$ is a maximal antichain in $\mathcal I^+$, then the set 
$\{X\in \mathcal I^+ : (\exists\, A\in \mathcal A) (X\subseteq^* A)\}$ is a dense--open subset of $\mathcal I^+$, and thus,   property $(p^w)$ can be formulated  in terms of antichains in the obvious way. Now, given a countable collection $\mathcal C$ of maximal antichains of the tree $\mathcal S$, since each antichain of $\mathcal S$ is countable, there is a level of the tree $\mathcal S$ which is above all the antichains in $\mathcal C$. Let $s$ be any element of $\mathcal S$ in this level, then for each antichain in the collection $\mathcal C$ we have that $\varphi(s)$ is below $\varphi(s')$ for a unique $s'$ in this antichain. Therefore, $\varphi(s)$ is a pseudo intersection of the collection of antichains of $\varphi[\mathcal S]$.

\medskip

Let us now show that the semiselective coideal $\mathcal I^+$ does not contain a p--point. Suppose, to reach a contradiction, that $\mathcal U$ is a p--point contained in $\mathcal I^+$. Then $\mathcal U$ has a filter base of sets of the form $\varphi(s)$ with $s\in \mathcal S$. If $s$ and $t$ are elements of $\mathcal S$ such that $\varphi(s)$ and $\varphi(t)$ are members of the filter base, then $s$ and $t$ are comparable in $\mathcal S$, since otherwise there would be almost disjoint elements of $\mathcal U$, which is impossible. Thus, the set $\mathcal C=\{s\in \mathcal S : \varphi(s) \text{  belongs to the filter base of } \mathcal U\}$  is a chain in $\mathcal S$.  Since  $\mathcal U$ is a p--point, any countable descending chain of elements of $\mathcal{U}$ can be extended in $\mathcal{U}$. Therefore, the chain $\mathcal C$ is not countable, but this contradicts that $\mathcal S$ is a Suslin tree.

\medskip

It remains to show how such an order embedding $\varphi: (\mathcal{S}, \leq) \rightarrow ([\omega]^\omega, ^*\!\supseteq)$ satisfying conditions $(1)$ and $(2)$ can be constructed. Using $\mathrm{CH}$, for every $A\in [\omega]^\omega$, let $\{ A=A_{0,\xi} \cup A_{1,\xi} :\xi \in\omega_1 \}$ be an enumeration of all the finite partitions of $A$ into two infinite subsets in which every such  partition is listed uncountably many times; similarly, for every $A\in [\omega]^\omega$, let  $\{ A= \bigcup_{j\in\omega} f^{A}_{j, \xi} : \xi\in\omega_1\}$ list all the infinite partitions of $A$ into finite pieces, each partition appearing uncountably many times in the list.

\medskip

The order embedding $\varphi: (\mathcal{S}, \leq) \rightarrow ([\omega]^\omega, ^*\!\supseteq)$ will be defined by induction on the levels of the Suslin tree $\mathcal S$. We can assume that in $\mathcal S$ each node has a countably infinite set of immediate successors; also, we will use that if $\alpha\in \beta\in \omega_1$, for every node $s$ at level $\alpha$ of the tree $\mathcal S$ there is a node $t$ at level $\beta$ of $\mathcal{S}$ such that $s<t$. 

\medskip 
 
If $r$ is the root of the tree $\mathcal S$, then we put $\varphi(r)=\omega$. Now, let $\alpha\in\omega_1$ and suppose $\varphi(s)$ has been defined for all $s\in \mathcal S$ of level $\leq \alpha$. List all the nodes belonging to the $\alpha$--th level of $\mathcal S$ as $\{t_n: n\in\omega\}$, so we will define $\varphi$ on the immediate successors of each node $t_n$ by induction on $n\in\omega$. 

\medskip

To define $\varphi$ on the immediate successors of $t_0$, we split $\varphi(t_0)$ into $\omega$--many pairwise disjoint infinite subsets $\{A_k : k\in\omega\}$, so that the following holds: for each $s\in\mathcal S$ with $s \leq t_0$, there are $k,k^\prime \in \omega$, with $k\neq k^\prime$, such that $A_k$ is contained in one of the parts of the partition $\varphi(s) = \varphi(s)_{0,\alpha} \cup \varphi(s) _{1,\alpha}$, and $A_{k^\prime}$ is a selector for the partition $\varphi(s)= \bigcup_{j\in\omega} f^{\varphi(s)}_{j, \alpha}$, that is, $|A_{k^\prime} \cap f^{\varphi(s)}_{j, \alpha}| \leq 1$ for each $j\in \omega$. This can be done as follows:

\medskip

List as $\{s_k : k\in\delta\}$, with $\delta \leq \omega$, the collection of all the nodes $s\in \mathcal S$ such that $s \leq t_0$; then, since $\varphi(t_0) \subseteq^* \varphi(s_0)$, there is an infinite set $A_0\subseteq \varphi(t_0)$ such that $\varphi(t_0)\setminus A_0$ is also infinite and either $A_0\subseteq \varphi(s_0)_{0,\alpha}$ or $A_0\subseteq \varphi(s_0)_{1,\alpha}$; likewise, there is an infinite set $A_1\subseteq \varphi(t_0)\setminus A_0$ such that $\varphi(t_0)\setminus (A_0\cup A_1)$ is also infinite and $A_1$ is a selector for the partition $\varphi(s_0)=\bigcup_{j\in\omega} f_{j, \alpha}^{\varphi(s_0)}$. Now, repeat the process for $s_1$ and $\varphi(t_0)\setminus (A_0\cup A_1)$ to obtain infinite sets $A_2$ and $A_3$ such that $A_2\subseteq \varphi(t_0)\setminus (A_0\cup A_1)$, $\varphi(t_0)\setminus (A_0\cup A_1\cup A_2)$ is infinite, and $A_2\subseteq \varphi(s_1)_{0,\alpha}$ or $A_2\subseteq \varphi(s_1)_{1,\alpha}$; also, $A_3\subseteq \varphi(t_0)\setminus (A_0\cup A_1\cup A_2)$, $\varphi(t_0)\setminus (A_0\cup A_1\cup A_2\cup A_3)$ is infinite, and $A_3$ is a selector for the partition $\varphi(s_1)= \bigcup_{j\in\omega} f^{\varphi(s_1)}_{j,\alpha}$. If the sequence $\{s_k: k\in \delta\}$ is infinite, continuing this way we get a family $\{A_k : k\in\omega\}$ of pairwise disjoint subsets of $\varphi(t_0)$; so, the $\omega$--many immediate successors of $t_0$ in the tree $\mathcal S$ are then sent by $\varphi$ to these $\omega$--many subsets of $\varphi(t_0)$. If the sequence $\{s_k: k\in \delta\}$ is finite, once we obtain the sets $A_i$ for $i\in 2\delta$ with $\delta <\omega$, partition the remaining infinite subset of $\varphi(t_0)$ into infinitely many infinite subsets to complete the collection $\{A_k : k\in\omega\}$.

\medskip
 
If $\varphi$ has been defined on the immediate successors of each node $t_0, \ldots, t_n$, then to define $\varphi$ on the immediate successors of $t_{n+1}$, we repeat the previous process with $t_{n+1}$ and the list $\{s_k: k\in \delta\}$, with $\delta\leq \omega$, consisting of all the nodes $s\in \mathcal S$ such that $s \leq t_{n+1}$ and $s\not\leq t_m$ for any $m\in \{0, \ldots, n\}$. 

\medskip 
 
Finally, let $\lambda\in\omega_1$ be a countable limit ordinal, and suppose $\varphi(s)$ has been defined for all $s\in \mathcal S$ of level $<\lambda$. Then, if $t$ is a node in the $\lambda$--th level of $\mathcal S$, define $\varphi(t)$ as a pseudo intersection of the countable collection $\{ \varphi(s) : s<t \}$, that is, $\varphi(t) \subseteq^{*} \varphi(s)$ for every $s<t$. In this way, the induction on the levels of the Suslin tree $\mathcal S$ is completed, and we have constructed an order embedding $\varphi: (\mathcal{S}, \leq) \rightarrow ([\omega]^\omega, ^*\!\supseteq)$ satisfying conditions $(1)$ and $(2)$.
\end{proof}

\smallskip

\begin{theorem} \label{coro_seminoselultra}
    Suppose $\mathrm{CH}$ holds together with $\neg \mathrm{SH}$. Then, there exists a semiselective coideal on $\omega$ that does not contain a selective ultrafilter. 
\end{theorem}

\begin{proof}
    This result is a straightforward consequence of Theorem \ref{seminoselultra} together with the well-known fact that every selective ultrafilter is a p--point.
\end{proof}

\section{Every semiselective coideal contains a selective ultrafilter}\label{everysemiselective}

In this section, we show that it is consistent with $\textrm{ZFC}$ that every semiselective coideal on $\omega$ contains a selective ultrafilter. For this section, some familiarity with the main concepts and ideas of forcing will be convenient, and we suggest the reader to see, for example, \cite{Bartoszynski-Judah, Halbeisen, Jech(Book), Kunen(Book), Shelah-pi}. 

\smallskip

\begin{theorem}\label{semiselcontainsselultra}
It is consistent with $\mathrm{ZFC}$ that for every semiselective ideal $\mathcal I$ on $\omega$, the coideal $\mathcal I^+$ contains a selective ultrafilter.
\end{theorem}

\smallskip

The proof of this theorem is given in Subsection \ref{consistency}, and we only describe here the line of the argument. If $\mathcal I$ is a semiselective ideal, forcing with $(\mathcal I^+, \subseteq^*)$ adds a selective ultrafilter contained in $\mathcal I^+$. Starting with $V=L$, we iterate forcing with semiselective coideals to reach a generic extension for which every semiselective coideal is used at some stage of the iteration to add  a selective ultrafilter contained in it. We have to make sure that this selective ultrafilter remains a selective ultrafilter until the end of the iteration. In the next subsection, we present the preservation results that we will use for this purpose.

\subsection{Some preservation results}

The notion of proper forcing was introduced by Shelah in \cite{Shelah}, and its study was continued in \cite{Shelah-pi}, where proofs of the preservation results stated below can be found. In particular that properness is preserved by countable support iterations, that is, if $\mathbb P_\alpha$ is a countable support iteration of $(\dot{\mathbb Q}_\beta : \beta<\alpha)$, and for each $\beta<\alpha$ we have $1_\beta\Vdash \dot{\mathbb Q}_{\beta}$ is proper, then $\mathbb P_\alpha$ is proper (see \cite[Theorem 2.7]{Abraham}). Also, $\omega^\omega$--boundedness is preserved (see \cite[Application 6.1]{Goldstern}) as well as p--points are preserved (see \cite[Theorem 4.1]{Blass-Shelah}).

\medskip

Next, we present some preservation results leading to Corollary \ref{corollary} that will be used in the proof of Theorem \ref{semiselcontainsselultra}.

\begin{proposition}[] 
(see Goldstern, \cite{Goldstern}). If $\mathbb P_\alpha$ is a countable support iteration of $(\dot{\mathbb Q}_\beta : \beta<\alpha)$, 
and for each $\beta<\alpha$
\begin{center} 
 $1_\beta\Vdash \dot{\mathbb Q}_\beta \text{ is proper and }\omega^\omega \text{--bounding}$,
\end{center}
then $\mathbb P_\alpha$ is proper and $\omega^\omega$--bounding.
\end{proposition}

\begin{proposition}[] \label{Blass-Shelah_theorem}
(Blass$-$Shelah, \cite{Blass-Shelah}). If $\mathbb P_\alpha$ is a countable support iteration of $(\dot{\mathbb Q}_\beta : \beta<\alpha)$, and for each $\beta<\alpha$ 
\begin{center}
 $ 1_\beta \Vdash \dot{\mathbb Q }_\beta \text{ is proper and preserves p--points}$,
\end{center}
then $\mathbb P_\alpha$ is proper and preserves p--points.
\end{proposition}

Let $\mathcal U$ be a selective ultrafilter in $V$ and let $\mathbb P$ be a forcing notion, we say that $\mathbb P$ preserves $\mathcal U$ if $\mathcal U$ generates a selective ultrafilter in any  $\mathbb P$--generic extension $V[G]$. The following result is used to prove the preservation of selective ultrafilters (see \cite[Lemma 21.12]{Halbeisen}).

\begin{proposition}[] \label{Halbeisen_theorem}
(see Halbeisen, \cite{Halbeisen}). If $\mathbb P$ is proper and $\omega^\omega$--bounding, and $\mathcal U$  is a selective ultrafilter in $V$ which generates an ultrafilter $\hat{\mathcal U}$ in the $\mathbb P$--generic extension $V[G]$, then $\hat{\mathcal U}$ is selective in $V[G]$.
\end{proposition}

As an immediate consequence of the preceding three propositions, the following preservation result can be deduced (see \cite[Corollary 21.13]{Halbeisen}):

\begin{corollary}[] \label{corollary}
(see Halbeisen, \cite{Halbeisen}). Let $\mathcal U$ be a selective ultrafilter in $V$, and let $\mathbb P_\alpha$ be a countable support iteration of $(\dot{\mathbb Q}_\beta: \beta<\alpha)$ such that for each $\beta<\alpha$ 
\begin{center}
   $1_\beta \Vdash \dot{\mathbb Q}_\beta \text{ is proper, } \omega^\omega\text{--bounding, and preserves } \mathcal U$, 
\end{center}
then $\mathbb P_\alpha$ preserves $\mathcal U$.
\end{corollary}

Furthermore, in the proof of Theorem \ref{semiselcontainsselultra}, we will also make use of the following preservation result (see \cite[Lemma 21.6]{Halbeisen}):

\begin{lemma}[] \label{lemma}
(see Halbeisen, \cite{Halbeisen}). 
Let $\mathbb P_{\omega_2}$ be a countable support iteration of $(\dot{\mathbb Q}_\beta : \beta <\omega_2)$ such that for each $\beta<\omega_2$
\begin{center}
$1_\beta \Vdash \dot{\mathbb Q}_\beta \text{ is a proper forcing notion of size} \leq\mathfrak{c}$.
\end{center}
If $\mathrm{CH}$ holds in the ground model, then $1_\beta\Vdash \mathrm{CH}$ for all $\beta<\omega_ 2$. 
\end{lemma}

It is well-known that if $\mathcal I$ is a semiselective ideal, forcing with $(\mathcal I^+, \subseteq^*)$ adds a selective ultrafilter contained in $\mathcal I^+$ (see \cite[Lemma 4.1]{Farah}). The semiselectivity of $\mathcal I$ implies that the  forcing $(\mathcal I^+, \subseteq^*)$ is $\sigma-$distributive, so it  does not add new reals (see \cite[Theorem 15.6]{Jech(Book)}).  The semiselectivity of $\mathcal I$ also gives that $(\mathcal I^+, \subseteq^*)$ is proper. To see this, let  $\theta$ be sufficiently large and  $(\mathcal I^+, \subseteq^*)\in M \preccurlyeq H(\theta)$ with $M$ countable. Let $\{\mathcal{D}_n \}_{n\in \omega}$ list all the dense--open subsets of $\mathcal I^+$. For $A\in \mathcal I^+$ in $M$ there is a diagonalization $B\in \mathcal I^+\!\restriction\! A$ of  $\{\mathcal{D}_n\}_{n\in \omega}$, then the condition $B$ is $(M,(\mathcal I^+, \subseteq^*))$--generic. Hence $(\mathcal I^+, \subseteq^*)$ is $\omega^\omega$--bounding and preserves p--points.

\subsection{The consistency result} \label{consistency}

What we are mainly concerned with here is the proof of Theorem \ref{semiselcontainsselultra}. First of all, we present some essential preliminaries related to a diamond principle (see, for example, \cite[Chapter 7]{Kunen(Book)}).

\medskip
 
Let $S^2_1=\{\alpha \in\omega_2 : \mathrm{cof}(\alpha) =\omega_1\}$ the stationary subset of all ordinals below $\omega_2$ of cofinality $\omega_1$. A sequence $\langle A_\alpha: \alpha \in S^2_1 \rangle$ of subsets of $\omega_2$ is a $\diamondsuit(S^2_1)$--sequence if $A_\alpha\subseteq \alpha$ for each $\alpha\in S^2_1$ and for every $A\subseteq \omega_2$ the set $\{\alpha\in S^2_1 : A\cap \alpha= A_\alpha\}$ is stationary. Denote by $\diamondsuit(S^2_1)$ the statement that a $\diamondsuit(S^2_1)$--sequence exists, and recall that $L\models \diamondsuit(S^2_1)$.

\medskip

Without further ado, we now proceed with the proof of Theorem \ref{semiselcontainsselultra}, which states the consistency with $\mathrm{ZFC}$ that every semiselective coideal contains a selective ultrafilter. 

\smallskip

\begin{proof}[Proof of Theorem \ref{semiselcontainsselultra}]
  
We work under $V=L$. We set up a countable support iteration $\mathbb P_{\omega_2}=(\mathbb P_\alpha, \dot{\mathbb Q}_\alpha : \alpha <\omega_2)$ such that for each $\beta\in S^2_1$ 
\begin{center}
   ${1}_\beta \Vdash \dot{\mathbb Q}_\beta \text{ is the forcing by a semiselective coideal}$,
\end{center}
so at stage $\beta\in S^2_1$, a selective ultrafilter contained in the semiselective coideal named by $\dot{\mathbb Q}_\beta$ is added.

\medskip

At each step of the iteration we have a model satisfying  $\textrm{CH}$, but $\textrm{CH}$ might not hold in the final model. For reasons related to the analysis of certain names for objects in the final model it will be convenient to have that  $\mathfrak{c} = \omega_2$ in the final model, and we can get this by cofinally iterating with Sacks forcing, which is proper, $\omega^\omega \text{--bounding}$, adds reals, and preserves $\textrm{CH}$ (see \cite[Chapter 23]{Halbeisen}). 

\medskip

For every ordinal $\beta<\omega_2$ of countable cofinality, that is $\mathrm{cof}(\beta)<\omega_1$, we use Sacks forcing, 
\begin{center}
  ${1}_\beta \Vdash \dot{\mathbb Q}_\beta \text{ is Sacks forcing}$,
\end{center}
and for other ordinals $\alpha$, let $\dot{\mathbb Q}_\alpha$ be a name for the trivial partial order $\{\emptyset\}$.

\medskip

We use the diamond principle $\diamondsuit(S^2_1)$ to get that every semiselective coideal in the $\mathbb P_{\omega_2}$--generic extension is used at some stage of the iteration at which a selective ultrafilter contained in it is added. Corollary \ref{corollary} can be applied to prove that this ultrafilter remains a selective ultrafilter until the iteration is completed. At each step of the iteration, by Lemma \ref{lemma} we have that $\textrm{CH}$ is preserved, and since reals are added in cofinally many steps, we get $\mathfrak{c} = \omega_2$ in the final generic extension.

\medskip
  
We want to show that if $G$ is $\mathbb P_{\omega_2}$--generic over $V$, in the generic extension $V[G]$ we have that every semiselective coideal contains a selective ultrafilter. For this we need first to analyze names for coideals in $V[G]$; in particular, we will see that names for coideals in $V[G]$ can be coded by subsets of $\omega_2$. 

\medskip
 
A coideal in $V[G]$, being a family of infinite subsets of $\omega$, can be coded by a subset of $\omega\times \omega_2$. Therefore, via this coding, every coideal in $V[G]$ has a name that is a nice name for a subset of $\omega\times \omega_2$. Fixing a bijection in the ground model between $\omega\times\omega_2$ and $\omega_2$, we can think that every coideal in $V[G]$ has a $\mathbb P_{\omega_2}$--name which is a nice name for a subset of $\omega_2$; in other words, it is of the form $\bigcup_{\beta<\omega_2} ( \{\beta\}\times C_\beta )$, where each $C_\beta$ is a maximal antichain of $\mathbb P_{\omega_2}$.

\medskip
 
To see that such a name can be coded by a subset of $\omega_2$, we analyze the sizes of maximal antichains in $\mathbb P_{\omega_2}$. It is well-known that any countable support iteration of length $\leq \omega_2$ of proper forcings of size $\aleph_1$ has the $\aleph_2$-cc (see \cite[Theorem 2.10]{Abraham}).

\medskip

We have then that a nice $\mathbb P_{\omega_2}$--name for a subset of $\omega_2$ uses at most $\omega_2$ conditions, the same holds for a  nice name for (the set coding) a coideal $\mathcal H$; so, it can be coded by a subset of $\omega_2$. Therefore, we can use a $\diamondsuit(S^2_1)$--sequence $\langle A_\alpha: \alpha \in S^2_1 \rangle$ to guess names for coideals in the generic extension $V[G]$.

\medskip

Let $\mathcal H\in V[G]$ be a semiselective coideal, and let $\dot{\mathcal H}$ be a $\mathbb P_{\omega_2}$--name for $\mathcal H$, and let $A\subseteq \omega_2$ be a set that codes $\dot{\mathcal H}$. Then, by $\diamondsuit(S^2_1)$ the set $S= \{\alpha\in S^2_1: A\cap \alpha=A_\alpha\}$ is stationary. 

\medskip

Take an ordinal $\theta$ large enough, for instance $\theta=\omega_3$, so that $L_\theta$ contains all that is relevant for the argument, namely $\mathbb P_{\omega_2}$, $\dot{\mathcal H}$, and $A$. Next, we consider a chain $\langle M_\alpha : \alpha <\omega_2\rangle$ such that: 
\setlist{nolistsep}
\begin{enumerate}
	\setlength{\itemsep}{0pt}
	\item[(1)] $M_0=\emptyset$.
	\item[(2)] For every $1<\alpha<\omega_2$: $M_\alpha\prec L_\theta$; $\mathbb P_{\omega_2}, \dot{\mathcal H}, A\in M_\alpha$; $|M_\alpha|=\omega_1$;  and  $[M_\alpha]^\omega\subseteq M_\alpha$ (so $\omega_1\subseteq M_\alpha$).
	\item[(3)] If $\alpha<\beta<\omega_2$, then $M_\alpha\in M_\beta$ and $M_\alpha\subseteq M_\beta$.
	\item[(4)] If $\lambda<\omega_2$ is a limit ordinal, then $M_\lambda=\bigcup_{\alpha<\lambda} M_\alpha$.
\end{enumerate}

\medskip

For any such structure $N=M_\alpha$ with $1<\alpha<\omega_2$,  all subsets of $\omega$ belong to $N$, and $\omega_2\cap N$ is an ordinal less than $\omega_2$. By elementarity, 
\begin{center}
	$N\models A \text{ codes } \dot{\mathcal H} \text{ and } \dot{\mathcal H} \text{ is a   $\mathbb P_{\omega_2}$--name for a semiselective coideal}$, 	
\end{center}
so $A\cap N$ codes a  $\mathbb P_{N\cap \omega_2}$--name for a semiselective coideal, namely $\dot{\mathcal H}\cap N$. Additionally, the collection $C= \{ M_\alpha \cap \omega_2 : \alpha<\omega_2\}$ is a closed unbounded subset of $\omega_2$.

\medskip

Let $\alpha\in S\cap C$, then  $A\cap \alpha=A_\alpha$ and  $A\cap \alpha$ codes a name for a semiselective coideal, and this name is $\dot{\mathcal H}$ restricted to some ordinal of the form $\omega_2\cap M_\beta$. At stage $\alpha$ of the iteration, $\dot{\mathbb Q}_\alpha$ is this name, and we forced with it adding a selective ultrafilter contained in $\mathcal H\cap V[G_\alpha]$. By Corollary \ref{corollary}, this selective ultrafilter remains a selective ultrafilter until the end of the iteration; therefore, in $V[G]$ it is a selective ultrafilter contained in $\mathcal H$. We have thus completed a proof of Theorem \ref{semiselcontainsselultra}.
\end{proof}

\smallskip

As a final comment, we mention that the following two questions about semiselective ideals containing selective ultrafilters remain open:

\begin{question} 
Can the hypothesis $\neg \mathrm{SH}$ be eliminated in Theorem \ref{coro_seminoselultra}? In other words, does $\mathrm{CH}$ imply that there are semiselective coideals that do not contain a selective ultrafilter?
\end{question}

\begin{question} 
In the iteration used to prove Theorem \ref{semiselcontainsselultra}, although the iterands at stages of uncountable cofinality add no new reals, in the final model the continuum hypothesis does not hold, so the following question arises: Is it consistent with $\mathrm{CH}$ that every semiselective coideal contains a selective ultrafilter?
\end{question}

\section*{Acknowledgments}

The first author was partially supported by the Fondo de Investigaciones de la Facultad de Ciencias de la Universidad de los Andes, grant INV-2021-127-2320. Also, this author thanks both the Fields Institute of the University of Toronto and the Centro de Ciencias Matem\'{a}ticas of the Universidad Nacional Aut\'{o}noma de M\'{e}xico Campus Morelia for their hospitality and partial support during the preparation of this article.

\smallskip

The second author thanks Joan Bagaria and the Mathematics Institute of the University of Barcelona for their hospitality and interesting conversations during the preparation of this article.

\smallskip

The third author was partially supported by the CONACyT grant A1-S-16164 and a PAPIIT grant IN 101323.

\smallskip

The authors thank the anonymous referee for his/her comments and suggestions, which improved the presentation of this paper.

\kern1em
\Addresses

\end{document}